\documentclass[10pt]{amsart}
\usepackage[active]{srcltx}
\usepackage{hyperref}
\usepackage{amssymb,amsthm}
\usepackage[all,arc,curve,color,frame]{xypic}
\usepackage{graphicx}
\usepackage{colortbl,array,tabularx}

\newcommand{\N}{\ensuremath{\mathcal{N}}}

\newcommand{\SSS}{\ensuremath{\mathcal{S}}}

\newcommand{\CC}{\ensuremath{\mathbb{C}}}

\newcommand{\NN}{\ensuremath{\mathbb{N}}}
\newcommand{\PP}{\ensuremath{\mathbb{P}}}

\newcommand{\QQ}{\ensuremath{\mathbb{Q}}}

\newcommand{\ZZ}{\ensuremath{\mathbb{Z}}}

\def\hol{{\mathcal{O}}}

\newtheorem{theorem}{Theorem}[section]
\newtheorem{lemma}[theorem]{Lemma}

\newtheorem{corollary}[theorem]{Corollary}
\newtheorem{prop}[theorem]{Proposition}

\theoremstyle{definition}
\newtheorem{df}[theorem]{Definition}
\newtheorem{notation}[theorem]{Notation}

\theoremstyle{remark}
\newtheorem{remark}[theorem]{Remark}

\numberwithin{equation}{section}

\renewcommand{\labelenumi}{(\arabic{enumi})}

\newcommand{\LL}{\mathcal{L}}
\renewcommand{\AA}{\mathbb{A}}
\newcommand{\set}[1]{\left\{ #1 \right \} }
\newcommand{\pare}[1]{\left ( #1 \right ) }
\newcommand{\ls}[1]{\left | #1 \right | }
\newcommand{\Span}[1]{\left < #1 \right > }

\newcommand{\OO}{\mathcal{O}}
\newcommand{\xx}{\mathbf{x}}

\renewcommand{\xx}{\mathbf{x}}
\newcommand{\yy}{\mathbf{y}}

\newcommand{\rt}{\rightarrow}
\newcommand{\drt}{\dashrightarrow}
\newcommand{\Aut}{\operatorname{Aut}}
\newcommand{\Sing}{\operatorname{Sing}}
\newcommand{\Proj}{\operatorname{Proj}}
\newcommand{\Fix}{\operatorname{Fix}}

\newcommand{\Sym}{\operatorname{S}}
\newcommand{\codim}{\operatorname{codim}}

\renewcommand{\S}{\operatorname{S}}
\newcommand{\Run}{R_\mathrm{un}}

\newcommand{\Va}{V_\mathrm{a}}

\renewcommand{\S}{\mathcal{S}}
\renewcommand{\H}{\mathcal{H}}

\newcommand{\Hom}{\operatorname{Hom}}

\begin{document}

\title [Unprojection and deformations of tertiary Burniat surfaces]{Unprojection and deformations \\ of tertiary Burniat surfaces}

\author{Jorge Neves}\address{Jorge Neves:  CMUC, Department of Mathematics,  University of Coimbra,
3001-454  Coimbra, Portugal}\email{neves@mat.uc.pt}


\author{Roberto Pignatelli}\address{Roberto Pignatelli: Dipartimento di Matematica, Universit\`a di Trento. Via Sommarive 14, loc. Povo, I-38050 Trento, Italy} \email{Roberto.Pignatelli@unitn.it}

\thanks{The authors are grateful to CIRM-Trento for
supporting the visit of the first author to Trento in April of 2010.
This work was partially supported by CMUC and FCT (Portugal) through the European program COMPETE/FEDER
and through Projects PTDC/MAT/099275/2008 and PTDC/MAT/111332/2009.
We are indebted with M. Mella for pointing out the theory of Enriques--Fano $3$-folds,
which inspired the construction in Section~\ref{sec: double cover}. 
The second author would like to thank I. Bauer and F. Catanese for some very interesting 
discussions about their work on Burniat surfaces.}


\date{\today}

\begin{abstract}
We construct a $4$-dimensional family of surfaces of general type with 
$p_g=0$ and $K^2=3$ and fundamental group $\ZZ/2 \times Q_8$, where $Q_8$ is the quaternion group. 
The family constructed contains the Burniat surfaces with $K^2=3$.
Additionally, we construct the universal coverings of the surfaces in our family as complete
intersections on $(\PP^1)^4$ and we also give an action of $\ZZ/2 \times Q_8$ on $(\PP^1)^4$
lifting the natural action on the surfaces.

The strategy is the following. We consider an \'etale $(\ZZ/2)^3$-cover $T$ of a surface with 
$p_g=0$ and $K^2=3$ and assume that it may be embedded in a Fano $3-$fold $V$. 
We construct $V$ by using the theory of parallel unprojection. Since $V$ is an Enriques--Fano $3$-fold, considering
its Fano cover yields the simple description of the universal covers above.

\end{abstract}

\maketitle

\section{Introduction}

A Burniat surface is the minimal resolution of singularities of a \emph{bidouble} cover, \emph{i.e.}, a
finite flat Galois morphism with Galois group $(\ZZ/2)^2$, of the projective plane branched along the divisors:
\[
D_1=A_1+A_2+A_3, \quad D_2=B_1+B_2+B_3, \quad  D_3=C_1+C_2+C_3,
\]
where $A_1$, $B_1$, $C_1$ form a triangle with vertices $\xx_1,\xx_2,\xx_3$, $A_1,A_2,A_3$
are lines through $\xx_1$, $B_1,B_2,B_3$ are lines through $\xx_2$ and $C_1,C_2,C_3$ are lines
through $\xx_3$. (\emph{Cf.}~Figure~\ref{fig: BurniatPicture}.) Burniat surfaces were first constructed
by Burniat \cite{burniat}, though a substantial part of the initial study of these surfaces was done,
about 10 years later by Peters \cite{peters}. They have an equivalent description known as the Inoue
surfaces \cite{inoue}, given as the quotient of a divisor in the product of three elliptic curves by a
finite group. See \cite{bcburniat1} for an excellent introduction to the subject of Burniat surfaces.
\medskip

Burniat surfaces are minimal surfaces of general type with
\mbox{$p_g=\dim H^0(\Omega^2)=0$} and hence with irregularity,
\mbox{$q=\dim H^0(\Omega^1)$}, equal to $0$.
The study of the mo\-du\-li space of surfaces of general type with these invariants
started in 1932 with Campedelli's celebrated construction of a surface of general type
with $p_g=0$ and $K^2=c_1^2=2$, as a double cover of the projective plane branched along a
curve of degree $10$ with $6$ singular points, not lying on a conic, all of type [3,3], that is a triple point
with another infinitely near triple point. Nowadays, this subject is
still the object of much attention, with new results on the description of whole components
of this moduli space (\emph{e.g}.~\cite{AP,CS,MP, MPR, PY1, PY2}) and on the proof of
existence of new ones (\emph{e.g}.~\cite{BCG, BCGP,BP,LP1,MPnewfam,NP1,PPS1}). See \cite{survey} for a survey
on surfaces of general type with $p_g=0$.

Let $S$ be a Burniat surface. If we assume that the branch divisors $D_1,D_2,D_3$ in
the configuration described earlier, besides satisfying the conditions stated there, are
otherwise general, then \mbox{$K^2_S=6$}. By the general theory of bidouble covers (see \cite{C}), 
imposing further, to the triple $D_1,D_2,D_3$, $m$ {\it singular points of type (1,1,1)} 
(which are points which belong to each $D_i$, which are smooth for each $D_i$, 
and such that the three tangent directions are different), then $K^2$ drops by $m$ and 
the other invariants do not change. This yields $6$ families (two for $m=2$: the family of nodal type
and the family of non nodal type) the dimensions of which are equal to $K_S^2-2=4-m$, respectively.

\begin{figure}[ht]
\begin{center}
\includegraphics{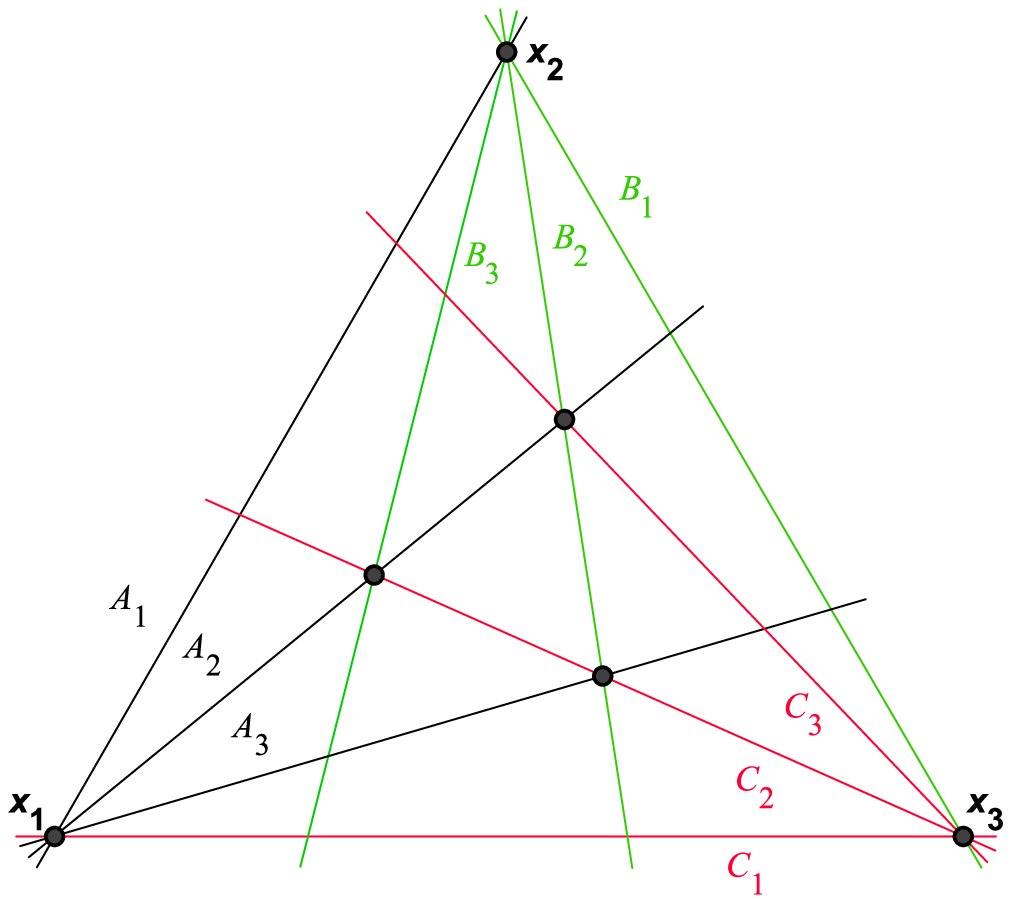}
\end{center}
\caption{Branching divisors for tertiary Burniat surfaces}\label{fig: BurniatPicture}
\end{figure}

Following Bauer and Catanese \cite{bcburniat1}, call a Burniat surface
\emph{primary} if $K^2_S=6$, \emph{secondary} if $K_S^2=4,5$, \emph{tertiary} if $K_S^2=3$ (see Figure~\ref{fig: BurniatPicture})
and \emph{quaternary} if $K_S^2=2$. 
From a certain point of view, what sets apart primary and secondary Burniat surfaces
from tertiary and quaternary Burniat surfaces is that the former families have
dimensions greater than or equal to the expected dimension $10\chi(\OO_S) -2K^2_S$ of the corresponding
moduli spaces, while the latter families have dimensions stricly less than the expected dimension of the corresponding moduli spaces.
More precisely, the family of tertiary Burniat surfaces is $1$-dimensional, whereas the moduli space of
surfaces of general type with $p_g=0$ and $K^2=3$ has expected dimension equal to $4$ and the family quaternary Burniat surfaces is
$0$-dimensional, whereas the moduli space of surfaces of general type with $p_g=0$ and $K^2=2$ (the Campedelli surfaces)
has expected dimension equal to $6$.

In 2001, Mendes Lopes and Pardini (\emph{cf.}~\cite{MPold}) proved that the $4$-dimensional fami\-ly of primary Burniat 
surfaces forms a normal, unirational, irreducible connected component of the moduli 
space of surfaces of general type with $p_g=0$ and $K^2=6$. In 2004, Kulikov (\emph{cf.}~\cite{Kul}) proved 
that the class of the quaternary Burniat surfaces belongs to the component of classical Campedelli surfaces,
\emph{i.e.}, $p_g=0$, \mbox{$K^2=2$} and torsion group $(\ZZ/2)^{3}$, which had been completely described  (\emph{cf.}~\cite{miyaoka,reid-campedelli}). In a deep recent analysis 
(\emph{cf}.~\cite{bcburniat1,bcburniat2,bcburniat3}), Bauer and Catanese have continued 
the study of the components of the moduli space of surfaces of general type containing the Burniat surfaces.
They gave an alternative proof of  Mendes Lopes--Pardini's result on primary Burniat surfaces. They showed that of the 
$3$ fami\-lies corresponding to secondary Burniat surfaces the one with $K^2=5$ and the one with $K^2=4$ of non nodal type
form irreducible connected components. They have also 
described the whole connected component containing the Burniat surfaces
with $K^2=4$ of nodal type, which turns out to have dimension $3$, one more than the expected dimension.

This article is devoted to a construction of a $4$-dimensional family of minimal
surfaces, $S$, of general type with $p_g(S)=0$ and $K_S^2=3$, containing, as a codimension $3$ subfamily,
the family of tertiary Burniat surfaces. We do this by constructing a $4$-dimensional family of surfaces
of general type $T$ with $\chi(\OO_T)=8$ and \mbox{$K^2_T=24$}, equipped with a free $G=(\ZZ/2)^3$ action.
We take $S$ as the quotient $T/G$. The family of surfaces $T$ is a linear subsystem of $\ls{-2K_V}$,
where $V$ is an Enriques--Fano $3$-fold in $\PP(1^7,2^8)$ obtained from
a complete intersection Fano $3$-fold in $\PP^6$ 
on which there exists an action of $G$ in\-du\-cing the action of this group on $T$. In this respect,
we can see $V$ as a \emph{key variety} for this construction;
just as weighted projective space acts as key variety in most elementary
constructions. This idea is reminiscent of the construction of a numerical Campedelli surface with torsion group $\ZZ/6$ of \cite{NP1}.
Lifting the action of $G$ to the Fano double cover of $V$ we obtain the simple description of our family
described in the next theorem, which synthesizes Theorem~\ref{thm: free action on $T$}, Theorem~\ref{thm: universal cover of $T$},
Theorem~\ref{thm: moduli} and Theorem~\ref{thm: The Burniat pencil} of this work.

\begin{theorem}\label{thm: main}
Consider $\PP^1\times \PP^1 \times \PP^1 \times \PP^1$, with coordinates $(t_{00},t_{01})$,
$(t_{10},t_{11})$, $(t_{20},t_{21})$, $(t_{30},t_{31})$ and the group $\tilde{G} < \Aut (\PP^1\times \PP^1\times \PP^1\times \PP^1)$
generated by the 3 automorphisms in the following table, where $\epsilon$ is a chosen square root of $-1$:
\begin{center}
\newcolumntype{x}[1]{>{\raggedleft\hspace{0pt}}p{#1}}
\centering \renewcommand{\arraystretch}{1.5}
\begin{tabular}[h]{x{.85cm}x{.85cm}x{.85cm}x{.85cm}x{.85cm}x{.85cm}x{.85cm}x{.85cm}x{.85cm}}
\noalign{\hrule height 1pt} %
& $t_{00}$ & $t_{01}$ & $t_{10}$ & $t_{11}$ & $t_{20}$ & $t_{21}$ & $t_{30}$ & $t_{31}$ \tabularnewline\noalign{\hrule height 0.25pt} 
$\tilde{\alpha}_1\tilde{\beta}_2$ & $-\epsilon t_{10}$ & $t_{11}$ &  $t_{00}$ & $\epsilon t_{01}$ & $t_{31}$ & $-\epsilon t_{30}$ & $t_{21}$ & $\epsilon t_{20}$ \tabularnewline\noalign{\hrule height 0.25pt}
$\tilde{\alpha}_2\tilde{\beta}_3$ &  $-\epsilon t_{20}$ &  $t_{21}$ & $t_{31}$ & $\epsilon t_{30}$ &  $t_{00}$ & $\epsilon t_{01}$ & $t_{11}$ & $-\epsilon t_{10}$ \tabularnewline\noalign{\hrule height 0.25pt}
$\tilde{\alpha}_3\tilde{\beta}_1$ &  $-\epsilon t_{30}$ &  $t_{31}$ & $t_{21}$ & $-\epsilon t_{20}$ & $t_{11}$ & $\epsilon t_{10}$ & $t_{00}$ & $\epsilon t_{01}$ \tabularnewline\noalign{\hrule height 1pt}
\end{tabular}
\end{center}
\medskip
\noindent
Then $\tilde{G} \cong \ZZ/2 \times Q_8$, where $Q_8$ denotes the standard quaternion group.
Consider also the $\tilde{G}$-invariant hypersurface of multi-degree $(1,1,1,1)$ given by
\[
Z_1:=(t_{01}t_{10}t_{20}t_{30}+t_{00}t_{11}t_{21}t_{31}=0)
\]
and the $\tilde{G}$-invariant surfaces $\tilde{T}$ cut out on $Z_1$ by the multi-degree $(2,2,2,2)$ hypersurfaces given by
\[
Z_2= \sum_{i=0}^3 \nu_i \left ( t_{i0}^2\prod_{j\neq i} t_{j1}^2+t_{i1}^2 \prod_{j\neq i} t_{j0}^2 \right ) - 2\nu_4\hspace{-.5cm}
\sum_{\scriptscriptstyle\text{$a+b+c+d$ even}} (-1)^{\frac{b+c+d-a}2} t_{0a}^2t_{1b}^2t_{2c}^2t_{3d}^2 =0
\]
for $\nu_0,\nu_1,\nu_2,\nu_3,\nu_4\in\CC$. Then, if the $\nu_i$
are general, $\tilde{G}$ acts freely on $\tilde{T}$
and the quotient $S=\tilde{T}/\tilde{G}$ is the canonical model of a
surface of general type with $p_g=0$, $K^2=3$ and $\pi_1(S)\cong \ZZ/2 \times Q_8$.
The family obtained in this way describes a $4$-dimensional locus in the moduli space of the surfaces of general type,
containing the tertiary Burniat surfaces, for which $-\nu_0=\nu_1=\nu_2=\nu_3$.
\end{theorem}
\noindent

Note that the fundamental group of tertiary Burniat surfaces has already been computed in 
\cite{bcburniat1}, which fixes a mistake in a previous computation in \cite{peters}.  
The study of surfaces of general type with $p_g=0$, $K^2=3$ and fundamental 
group of order $16$ is of special interest as, according to a conjecture of 
M.~Reid, this number should be the maximum order of their (algebraic)
fundamental groups.

\medskip

\noindent
Our construction gives a $4-$dimensional stratum of the moduli space of the surfaces of general type 
containing the tertiary Burniat surfaces. In \cite{bcburniat3}, Bauer and Catanese prove that the 
irreducible component of the moduli space of surfaces of general type containing the 
tertiary Burniat surfaces has dimension 4, and they construct a proper open set of it; it follows
that also our family forms an open set of the same component. 
We expect that our family is not a proper subset, covering the full irreducible component. It is also 
reasonable to guess that this irreducible component is a full connected component of the moduli space.

\medskip

We now explain the motivation for our construction.
Let $T$ be a minimal regular surface of general type with $\chi(\OO_T)=8$ and \mbox{$K_T^2=24$}. Assume that  $T\in\ls{-2K_V}$, 
where $V$ is a $\QQ$-Fano $3$-fold with $n$ singular points of type $\frac{1}{2}(1,1,1)$. Then $h^0(-K_V)=p_g=7$,
 $-K_V^3=K^2_T/2=12$ and by the orbifold Riemann--Roch formulas (\emph{cf.}~\cite{ABR,BS}), 
$4p_g=K^2_T+12-n$, \emph{i.e.}, $n=8$. This leads to a candidate $3$-fold $V$ anticanonically embedded in
$\PP(1^7,2^8)$ that, by the Graded Ring Database \cite{Br}, projects to a complete intersection $W_{2,2,2}\subset \PP^6$.
On the other hand, suppose that $T$ is equipped with a free \mbox{$G=(\ZZ/2)^3$} action.
By the Lefschetz Holomorphic Fixed Point Formula we know
the character of the representation of $G$ on $H^0(nK_T)$. 
Throughout the paper $a,b,c,d$ vary in $\ZZ/2=\{0,1\}$, and we wil use the notation $0'=1$ and $1'=0$. 
Writing $\chi_{abc}$, for the irreducible representations of $G$, we get:
\begin{equation}\label{eq: L194}
\renewcommand{\arraystretch}{1.4}
\begin{array}{c}
H^0(K_T) = \bigoplus_{(a,b,c)\in G\setminus\set{(0,0,0)}} \chi_{abc}, \quad  H^0(2K_T) =  \bigoplus_{(a,b,c)\in G} \chi^{\oplus 4}_{abc}, \\
\Sym^2 H^0(K_T) = \chi_{000}^{\oplus 7} \oplus \chi_{100}^{\oplus 3}\oplus\chi_{010}^{\oplus 3}\oplus\chi_{001}^{\oplus 3}\oplus\chi_{110}^{\oplus 3}\oplus\chi_{101}^{\oplus 3}\oplus\chi_{011}^{\oplus 3}\oplus\chi_{111}^{\oplus 3}.
\end{array}
\end{equation}
We deduce that the canonical ring of $T$,
\[
R(T,K_T)=\bigoplus_{n\in \NN} H^0(nK_T),
\]
on which $G$ acts, has $3$ invariant quadric relations and needs $7$ new
gene\-ra\-tors in degree $2$, one for each of the nontrivial rank $1$ representations $G$.
This agrees with the properties of $V$. The anticanonical ring $R(V,-K_V)$ has $8$ generators
of degree $2$ and $3$ quadric relations between the degree $1$ generators, coming from the defining equations of $W_{2,2,2}\subset \PP^6$.
Note that $R(T,K_T)$ can be obtained from $R(V,-K_V)$ by taking a quotient by a degree $2$ regular element.
\medskip

\noindent
As in \cite{NP1}, the first goal is to construct $V$ from $W_{2,2,2}\subset \PP^6$ using parallel unprojection, which is to say,
unproject all at once $8$ divisors in $W$ satisfying sufficiently general conditions.
\medskip

\noindent
The second goal is to set up an action of $G\cong(\ZZ/2)^3$ on $\PP(1^8,2^8)$ that leaves $Y$, $V$ and $T$ invariant
and is fixed point free on $T$. With this in mind we establish a $(\ZZ/2)^6$ action on $\PP(1^8,2^8)$ which leaves
$Y$ and $V$ invariant and for which there exists a subgroup $H\subset (\ZZ/2)^6$ isomorphic to $(\ZZ/2)^5$ which
leaves $T$ invariant. We then show that $H$ has a subgroup $G\cong (\ZZ/2)^3$ which acts fixed point freely on $T$. The upshot
is that the quotient group $H/G\cong (\ZZ/2)^2$ acts on $S:=T/G$ and the quotient map coincides
with the bicanonical map of $S$. (\emph{Cf}.~Proposition~\ref{prop: the bicanonical map of S}.)

\medskip

The paper is divided up as follows. In Section~\ref{sec: construction} we describe the construction of
$Y\subset \PP(1^8,2^8)$ via parallel unprojection of a $4$-fold complete intersection of $3$ quadrics
$X\subset \PP^7$ using the format introduced in \cite{NP2}. We obtain a $\QQ$-Fano $3$-fold $V\subset Y$
by taking a hypersurface section of degree $1$ of $V$ and the surface $T\subset V$ by taking a hypersurface
section of degree $2$ of $V$. The bulk of this section is concerned with the study of the geometry of $V$ (with
emphasis on its singularities) and setting up of the group action described above.
In Section~\ref{sec: double cover}, we show that $Y$ is the quotient of $\PP^1\times\PP^1\times\PP^1\times \PP^1$
by an involution and we lift the action of $G$ to an action of $\tilde{G}= \ZZ/2 \times Q_8$ on 
$\PP^1\times\PP^1\times\PP^1\times \PP^1$.
We obtain a description of our surfaces as quotient by a fixed point free action of $\tilde{G}$ on
a complete intersection in $\PP^1\times\PP^1\times\PP^1\times \PP^1$, which enables the computation of
their fundamental group. Finally we show that the family constructed is unirational and has $4$ moduli. 
In Section~\ref{sec: bicanonical}
we carry out a detailed study of the bicanonical map of $S=T/G$. We show that the bicanonical map is a bidouble cover of a
singular cubic surface $S_3\subset \PP^3$ and compute the branch loci of this map. Via a birational map $S_3\drt \PP^2$ we
reinterpret this bidouble cover as a bidouble cover of $\PP^2$ and use it to show that the family of surfaces constructed
contains the family of tertiary Burniat surfaces.

\section{The Construction of $S$}\label{sec: construction}

Consider $\PP^7$ with homogeneous coordinates $x_{00}$, $x_{01}$, $x_{10}$, $x_{11}$, $x_{20}$, $x_{21}$, $x_{30}$, $x_{31}$ and
let $X\subset \PP^7$ be the $4$-fold complete intersection of $3$ quadrics given by:
\begin{equation}\label{eq: quadrics}
x_{00}x_{01}=x_{10}x_{11}=x_{20}x_{21}=x_{30}x_{31}.
\end{equation}
Notice that $X$ contains the $16$ linear spaces given by:
\begin{equation}\label{eq:all the planes}
H_{abcd}=(x_{0a}=x_{1b}=x_{2c}=x_{3d}=0),\quad  a,b,c,d\in \set{0,1}
\end{equation}
all of which have codimension $1$ in $X$. These $16$ linear spaces 
can be thought of as the vertices of the $4$-cube, 
by identifying their equations (\ref{eq:all the planes}) with the vertex $(a,b,c,d)$.
An edge between two vertices means that the intersection of the corresponding linear spaces
has dimension $\geq 2$ or, equivalently, that the union of the sets of equations of the linear spaces 
does not contain a regular sequence of length $6$.

Since the homogenous coordinate rings of $X$ and of each linear space are Gorenstein
graded rings, we can use Kustin--Miller parallel unprojection
on a subset of the set of linear spaces in (\ref{eq:all the planes}).
Indeed the format of the equations of $X$ was \mbox{studied} in \cite[Section 3]{NP2},
where a sufficient condition for the existence of the parallel unprojection was given.
In our case, a subset of linear spaces can be unprojected if the defining 
equations of any two linear subspaces in it contain a regular sequence of length $6$.
Since the $4$-cube is a bipartite graph, there are $2$ maximal subsets with this property. These subsets 
yield isomorphic constructions, thus we shall fix one.  
Let $\LL$ denote the subset of $\set{0,1}^4$ consisting of the $4$-tuples with even sum and consider
the corresponding subset of linear spaces: $\set{H_{abcd}  \mid (a,b,c,d)\in \LL }$.
Recall that throughout the paper we shall be using the following shorthand notation: $0'=1$ and $1'=0$.
\begin{remark}\label{remark: SingX}
Notice that $H_{abcd} \cap H_{a'b'c'd'}=\emptyset$. Any other pair of distinct elements in $\{H_{abcd}  \mid (a,b,c,d)\in \LL \}$
intersect along a line. These $24$ lines form the singular locus of $X$.
\end{remark}

According to \cite[Lemma~3.2]{NP2} we can perform the parallel unprojection of these
$8$ linear spaces in $X$, to obtain a projectively Gorenstein
subscheme of a weighted projective space, $Y\subset \PP(1^8,2^8)$, as follows.

\begin{df}\label{def: the unprojection and projection maps}
Consider, for each $(a,b,c,d)\in \LL$ the rational section of $\hol_X(2)$
\begin{equation}\label{eq: definition of y_abcd}
\renewcommand{\arraystretch}{1.7}
\begin{array}{c}
\displaystyle \varphi_{abcd} := \frac{x_{1b'}x_{2c'}x_{3d'}}{x_{0a}} =
\frac{x_{0a'}x_{2c'}x_{3d'}}{x_{1b}} = \frac{x_{0a'}x_{1b'}x_{3d'}}{x_{2c}}  = \frac{x_{0a'}x_{1b'}x_{2c'}}{x_{3d}},
\end{array}
\end{equation}
where the equalities follow from (\ref{eq: quadrics}). 
The divisor of the poles of $\varphi_{abcd}$ is $H_{abcd}$.

We denote by  $\varphi \colon X \drt \PP(1^8,2^8)$ the unprojection map, \emph{i.e.}, 
the rational map  
$$\varphi(x_{00},x_{01},\dots,x_{31})=(x_{00},x_{01},\dots,x_{31},\varphi_{0000}(x_{ia}),
\dots,\varphi_{1111}(x_{ia})).$$ We define $Y:=\varphi(X)$.
\end{df}

\begin{notation}\label{not: pi and y}
We denote accordingly the weight $2$ variables of the ambient weighted projective space by 
$y_{abcd}$: $y_{abcd}$ is 
the variable corresponding to $\varphi_{abcd}(x_{ia})$ in the definition of $\varphi$.
Let $\pi \colon \PP(1^8,2^8)\drt \PP^7$ denote the projection map, \emph{i.e.}, the rational map obtained by forgetting the degree $2$ variables.
\end{notation}

The ideal $J$ of the subvariety $Y\subset \PP(1^8,2^8)$ is generated by the following homogenous
polynomials: the original $3$ quadrics --- given by the difference of two terms in
(\ref{eq: quadrics}) --- $32$ cubics,
given by
\begin{equation}\label{eq: cubic relations}
\begin{array}{c}
\displaystyle y_{abcd}x_{0a}-x_{1b'}x_{2c'}x_{3d'},\quad  y_{abcd}x_{1b}-x_{0a'}x_{2c'}x_{3d'},\\
\displaystyle y_{abcd}x_{2c}-x_{0a'}x_{1b'}x_{3d'},\quad  y_{abcd}x_{3d}-x_{0a'}x_{1b'}x_{2c'},
\end{array}
\end{equation}
for every $(a,b,c,d)\in\LL$; and $28$ quartics, given by
\begin{equation}\label{eq: quartic relations}
y_{a_0a_1a_2a_3}y_{b_0b_1b_2b_3} - \frac{x_{0a_0'}x_{1a_1'}x_{2a_2'}x_{3a_3'}}{x_{ia_i'}x_{ia_i}}
\cdot \frac{x_{0b_0'}x_{1b_1'}x_{2b_2'}x_{3b_3'}}{x_{jb_j'}x_{jb_j}}
\end{equation}
for every distinct $(a_0,a_1,a_2,a_3),(b_0,b_1,b_2,b_3)\in \LL$, where, given $(a_0,a_1,a_2,a_3)$, $(b_0,b_1,b_2,b_3)$ in $\LL$,  $i$ and $j$ are such that $a_i\not = b_i$ and $a_j\not = b_j$, so that the fractional expression of (\ref{eq: quartic relations}) is always a polynomial.

\begin{remark}\label{rmk: unprojection map and iso between open}
The unprojection map $\varphi \colon X\drt Y$ is a birational map between $X$ and $Y$, with inverse $\pi_{|Y} \colon Y \drt X$. Indeed, $\varphi$ induces an isomorphism
\begin{equation}\label{eq: isomorphism between Xminus and Yminus}
X\setminus \pare{\bigcup\nolimits_{abcd\in \LL} H_{abcd}} \rt Y\setminus \pare{\bigcup\nolimits_{abcd\in \LL} \H_{abcd}},
\end{equation}
where $\H_{abcd}$ is the subscheme of $Y$ given by $x_{0a}=x_{1b}=x_{2c}=x_{3d}=0$.
\end{remark}

\begin{notation}\label{not: notation xx, yy, lijab, Sijab}
Firstly we make notation for the coordinate points of $\PP^7$ and $\PP(1^8,2^8)$.
Given $0\leq i\leq 3$ and $a\in \set{0,1}$ we denote by
$\xx_{ia}$ the point of $\PP^7$, or of $\PP(1^8,2^8)$, depending on the context, having all but the coordinate $x_{ia}$ equal to zero. Similarly, given $(a,b,c,d)\in \LL$,
we denote by $\yy_{abcd}\in \PP(1^8,2^8)$ the point defined in an analogous way. Note that the 8 points $\yy_{abcd}$ are the intersection of $Y$ with the singular locus of the ambient space, 
and also the centers of the projection  $\pi_{|Y}$.
Secondly we establish notation for a distinguish set of surfaces in $\PP(1^8,2^8)$.
There are $24$ quartic polynomials in (\ref{eq: quartic relations}) involving
the product of $2$ squares. Such is the case with $y_{0011}y_{0000}-
x_{01}^2x_{11}^2$. This polynomial defines a subscheme, $\S_{11}^{01}$, of dimension $2$ of the
$3$-dimensional projective space $\PP(1^2,2^2)$ with variables $x_{01}$, $x_{11}$, $y_{0011}$, $y_{0000}$ that
we can regard as a subscheme $\S_{11}^{01}\subset\PP(1^8,2^8)$, by setting all but the coordinates
$x_{01},x_{11},y_{0011},y_{0000}$ equal to $0$.
Similarly, given $0\leq i<j \leq 3$ and $a,b\in \set{0,1}$ we denote by $\S_{ab}^{ij}$ the subscheme of $\PP(1^2,2^2)\subset \PP(1^8,2^8)$ defined by the quartic polynomial of (\ref{eq: quartic relations}) involving $x_{ia}^2x_{jb}^2$.
These are $24$ surfaces contained in $Y$.
\end{notation}

\begin{lemma}\label{lemma: bunch of Sijab and Habcd} Set-theoretically,
$\H_{a'b'c'd'}=\set{\yy_{abcd}}\cup \S^{01}_{ab}\cup \S^{02}_{ac}\cup \S^{03}_{ad}\cup \S^{12}_{bc}\cup \S^{13}_{bd}\cup \S^{23}_{cd}$. In particular $\H_{abcd}$ is $2$-dimensional, for all $(a,b,c,d)\in \LL$.
\end{lemma}

\begin{proof}
We prove the lemma for $(a,b,c,d)=(0,0,0,0)$. The proof for the remaining $(a,b,c,d)\in\LL$ is similar.
Comparing the definitions of $\H_{abcd}$ in Remark~\ref{rmk: unprojection map and iso between open} and
of $\S_{ab}^{ij}$ and  $\yy_{abcd}$ of Notation~\ref{not: notation xx, yy, lijab, Sijab} it follows that
\[
\S^{01}_{11}\cup \S^{02}_{11}\cup \S^{03}_{11}\cup \S^{12}_{11}\cup \S^{13}_{11}\cup \S^{23}_{11}\cup \set{\yy_{1111}}\subset \H_{0000}.
\]
Conversely, let $\xx \in \H_{0000}$. From the cubic equations (\ref{eq: cubic relations}) involving $y_{0000}$, we see that there exist distinct $i,j \in \set{0,1,2,3}$ such that $x_{i1}=x_{j1}=0$.
Assume that $i=0$ and  $j=1$. If $y_{abcd}=0$, for all $(a,b,c,d)\in\LL \setminus \set{(0,0,0,0),(1,1,0,0)}$,
then $y_{0000}y_{1100} -x_{21}^2x_{31}^2=0$ is the only equation of $Y$ not made trivial.
In this situation $\xx \in \S_{11}^{23}$. Suppose that $y_{abcd}\not =0$ for some
$(a,b,c,d)\in\LL \setminus \set{(0,0,0,0),(1,1,0,0)}$.
Then, from the quartic equations (\ref{eq: quartic relations}) involving $y_{abcd}$ we see that all other weight $2$ variables are zero and, using the cubic equations (\ref{eq: cubic relations})  involving $y_{abcd}$,
that \mbox{$x_{2c}=x_{3d}=0$}. Note that necessarily $(c,d)\neq (0,0)$.
Now, if $(c,d)=(1,0)$ then all variables but $y_{ab10}$ and $x_{31}$ vanish.
In this case, either $(a,b)=(0,1)$ and $\xx \in \S_{11}^{03}$, or $(a,b)=(1,0)$ and $\xx \in\S_{11}^{13}$.
Similarly, if $(c,d)=(0,1)$, $\xx \in \S_{11}^{02} \cup \S_{11}^{12}$. Finally, if $(c,d)=(1,1)$ then,
$\xx = \yy_{1111}$ or $\xx=\yy_{0011}$, and we conclude by observing that $\yy_{0011} \in \S^{01}_{11}$.
The same reasoning applies for any other distinct $i,j\in \set{0,1,2,3}$.
\end{proof}

\begin{prop}\label{prop: Y reduced and irreducible normal...}
$Y$ is a reduced and irreducible normal $4$-dimensional subscheme of $\PP(1^8,2^8)$. Moreover 
$K_Y = \OO_Y(-2)$ and $\deg Y=\deg X+4 = 12$.
\end{prop}

\begin{proof}
Let $R$ denote the coordinate ring of $X$. The fact that $\dim Y = 4$ is a
consequence of the fact that $\dim \Run = \dim R = 4$, coming from the general
theory of Kustin--Miller unprojection. However it is also a consequence of the
isomorphism (\ref{eq: isomorphism between Xminus and Yminus}) and Lemma~\ref{lemma: bunch of Sijab and Habcd}.
$\Run$ is obtained as an unprojection of R, that has canonical
module equal to $R(-2)$. Hence $\Run$ is Gorenstein and has a canonical module equal to $\Run(-2)$, \emph{cf}.~\cite{NP2}.
In view of Remark~\ref{remark: SingX}, isomorphism (\ref{eq: isomorphism between Xminus and Yminus}) 
and Lemma~\ref{lemma: bunch of Sijab and Habcd},
$\codim \Sing Y \geq 2$. Since $\Run$ is Cohen--Macaulay we deduce that $\Run$ is
a normal domain, \emph{cf}.~\cite[Theorem~18.15]{Ei}. Hence $Y$ is a reduced and irreducible normal subscheme of $\PP(1^8,2^8)$. That $K_Y=\OO_Y(-2)$
follows from the computation of the canonical module of $\Run$. By \cite[Proposition~3.4]{NP2},
$\deg Y=\deg X+4 = 12$.
\end{proof}

We can now define the key variety $V$. This variety is obtained intersecting $Y$ with the
hypersurface given by $x_{00}+x_{01}=0$. The reason for this choice of degree $1$ polynomial
will be clear from the action of $G\cong (\ZZ/2)^{\oplus 3}$ on $V$ that we describe below.
We will regard $V$ as a subvariety of $\PP(1^8,2^8)$ defined by the ideal
$J+(x_{00}+x_{01})$, \emph{i.e.},
the ideal generated by $x_{00}+x_{01}$ and the polynomials in (\ref{eq: quadrics}), (\ref{eq: cubic relations}) and
(\ref{eq: quartic relations}). Since $x_{00}+x_{01}$ is a regular element of $\Run$ and this ring
is Cohen--Macaulay we deduce that $V$ is a $3$-fold of degree $12$. Clearly, $V$
is the parallel unprojection of the $8$ planes $\Pi_{abcd}:=H_{abcd}\cap (x_{00}+x_{11}=0)$ in the
$3$-fold $W=X\cap(x_{00}+x_{01}=0)$. The following diagram shows the construction so far.
\begin{equation}
\xymatrix{
V\ar@{-->}[d]\ar@<-.6ex>@{^{(}->}[r]& Y\subset \PP(1^8,2^8)\ar@<-6ex>@{-->}^{\pi_{|Y}}[d]\\
W\ar@<-.6ex>@{^{(}->}[r]& X\subset \PP^7\phantom{Aaaa..}\\
}
\end{equation}

\begin{prop}\label{prop: singularities of V}
The singular locus of $V=Y\cap(x_{00}+x_{01}=0)$ consists of
$14$ points, $8$ quotient singularities of type $\frac{1}{2}(1,1,1)$
at the points $\yy_{abcd}$ and $6$ isolated singularities locally analytically
isomorphic to the vertex of a cone over the Del Pezzo surface
$\PP^1 \times \PP^1 \subset \PP^8$ at the points $\xx_{ia}\in
\PP(1^8,2^8)$, for $i>0$.
\end{prop}

\begin{proof}
Consider $W=X\cap(x_{00}+x_{01}=0)$.
The variety $W$ is smooth away from $\cup_{abcd\in \LL} H_{abcd}$.
Since $\pi$ is an isomorphism away from $\cup_{abcd\in \LL} \H_{abcd}$
(\emph{cf.}~Remark~\ref{rmk: unprojection map and iso between open}), we deduce that
\begin{equation}\label{eq: L572}
\Sing(V)\subset V \cap (\cup_{abcd\in \LL} \H_{abcd})=V\cap (\cup \S^{ij}_{ab}).
\end{equation} We start by analyzing the points $\xx$
of $\Sing(V)$ in the locus $\set{\yy_{abcd} \mid (a,b,c,d)\in \LL}$.
We assume, without loss of generality, that $\xx= \yy_{0000}$. Consider the affine piece of $V$ given by $y_{0000}=1$.
Then, using the quartic equations (\ref{eq: quartic relations}) we can eliminate all of the remaining $y_{abcd}$,
using the cubic equations (\ref{eq: cubic relations}) we can eliminate $x_{00},x_{10},x_{20},x_{30}$ and using
$x_{00}+x_{01}$ we can eliminate $x_{01}$. The coordinates
$x_{11}$, $x_{21}$, $x_{31}$, $y_{0000}$ map an analytic neighborhood of $\yy_{0000}\in V$ isomorphically
onto a neighborhood of the point $(0,0,0,1)\in\PP(1^3,2)$, which is a quotient singularity of type
$\frac{1}{2}(1,1,1)$.
\smallskip

\noindent
Suppose now that $\xx \in \Sing(V)\setminus \set{\yy_{abcd} \mid (a,b,c,d)\in \LL}$. Let $\Va\subset \AA^{16}$ denote the affine cone of $V$. Among the equations of $\Va$, besides $x_{00}+x_{01}=0$, we find the $7$ quartic equations $y_{abcd}y_{0000}-\cdots=0$, plus
\[\renewcommand{\arraystretch}{1.3}
\begin{array}{c}
y_{0000}x_{00} -x_{11}x_{21}x_{31}=0, \quad y_{0000}x_{10} -x_{01}x_{21}x_{31}=0, \\
y_{0000}x_{20} -x_{01}x_{11}x_{31}=0, \quad y_{0000}x_{30} -x_{01}x_{11}x_{21}=0.
\end{array}
\]
Let us take the $12\times 12$ minor of the Jacobian matrix of the ideal defining $\Va$ of the gradients of these $12$ polynomials with respect to the variables $x_{01}$, $y_{abcd}$ for $(a,b,c,d)\in \LL\setminus\set{(0,0,0,0)}$ and $x_{00},x_{10},x_{20},x_{30}$. This minor is equal to $\pm y_{0000}^{11}$, where the sign depends on the order we give to the equations and to the variables. Similarly we can find minors of the form $\pm y^{11}_{abcd}$, for all $(a,b,c,d)\in \LL$. Hence if $\xx \in  \Sing(V)\setminus \set{\yy_{abcd} \mid (a,b,c,d)\in \LL}$ then $y_{abcd}=0$, for all $(a,b,c,d)\in \LL$. From (\ref{eq: L572}) and Lemma~\ref{lemma: bunch of Sijab and Habcd}, we deduce
$\xx \in \set{\xx_{10},\xx_{11},\xx_{20},\xx_{21},\xx_{30},\xx_{11}}$.
We assume, without loss of generality, that $\xx = \xx_{10}$. Consider the affine piece of $Y$ given by
$x_{10}=1$. Here, we can use the cubic equations (\ref{eq: cubic relations}) to eliminate all
variables of the form $y_{a0cd}$ and one of the quadrics (\ref{eq: quadrics}) to
eliminate $x_{11}$. After eliminating these $5$ variables, we see that this affine piece
of $Y$ is isomorphic to the subvariety of $\AA^{9}$ defined by the $2\times 2$
minors of the symmetric matrix
\begin{equation*}
\begin{pmatrix}
y_{1100}&
x_{31}&
x_{21}&
x_{00}\\
&
y_{0110}&
x_{01}&
x_{20} \\
&&
y_{0101}&
x_{30}\\
\operatorname{sym}&&&
y_{1111}
\end{pmatrix},
\end{equation*}
with $\xx_{10}$ being identified with the origin of $\AA^{9}$. Hence $\xx_{10}$ is a singular point
of $Y$ locally isomorphic to the cone over the $2$-Veronese embedding of $\PP^3$ in $\PP^9$.
Since $V=Y\cap (x_{00}+x_{01}=0)$ we conclude that $V$ is locally, near $\xx_{10}$, analytically  isomorphic to
a cone over the Del Pezzo surface $\PP^1\times \PP^1 \subset \PP^8$. Similarly for all other points in $\set{\xx_{10},\xx_{11},\xx_{20},\xx_{21},\xx_{30},\xx_{31}}$.
\end{proof}

\begin{corollary}
$V$ is a reduced and irreducible normal $3$-dimensional subscheme of $\PP(1^8,2^8)$.
Moreover $K_V = \OO_V(-1)$ and $\deg(V)= 12$.
\end{corollary}

\begin{proof}
The proof is similar to that of Proposition~\ref{prop: Y reduced and irreducible normal...}.
\end{proof}

The surface $T$, on which we will set up a group action of $G\cong (\ZZ/2)^3$
will be a suitable hypersurface section of $V$ of degree $2$, and therefore a canonical surface. In particular
the group action is induced by action of $G$ on the ambient weighted projective space.
What we do next is to set an action of the larger group $(\ZZ/2)^6$ on the ambient space, which leaves $V$ invariant.
Following that, we single out a subgroup $G\cong (\ZZ/2)^3$ of $(\ZZ/2)^6$ inducing on $H^0(\OO_V(1))$ the
regular representation of $G$ minus the trivial rank $1$ representation.
Finally, we choose the surface \mbox{$T\in \ls{\OO_V(2)}$} in such a way that $G$ leaves it
invariant and that the induced representation of $G$ on $H^0(\OO_T(2))=H^0(K_T)$ is the sum of $4$
copies of the regular representation.
\medskip

\noindent
Let $\alpha_1$, $\alpha_2$, $\alpha_3$, $\beta_1$, $\beta_2$, $\beta_3$ be generators of $(\ZZ/2)^6$. Let them
act on the space $\Span{x_{ij}}$ in the following way: $\alpha_i$ exchanges $x_{00}$ with $x_{01}$ and exchanges
$x_{i0}$ with $x_{i1}$, fixing all the remaining variables; $\beta_i$ takes $x_{i0}$ to $-x_{i0}$ and $x_{i1}$ to $-x_{i1}$,
fixing all the remaining variables.
Since the actions of two generators commute, we obtain an action of $(\ZZ/2)^6$ on $\PP^7$.
Clearly, by inspection of (\ref{eq: quadrics}), $X$ is invariant under this action.
The identification of the variables $y_{abcd}$ with the rational functions on $X$ of (\ref{eq: definition of y_abcd})
induces an extension of this action to $\PP(1^8,2^8)$
so that $Y$, and $V$ as well, become invariant. Since
\begin{equation}\label{eq: extending the action}
\varphi_{abcd} = \frac{x_{1b'}x_{2c'}x_{3d'}}{x_{0a}} \stackrel{\alpha_1}{\longrightarrow}\frac{x_{1b}x_{2c'}x_{3d'}}{x_{0a'}} = \varphi_{a'b'cd},
\end{equation}
\emph{etc}., it suffices to set $\alpha_1(y_{abcbd})=y_{a'b'cd}$,
$\alpha_2(y_{abcd})=y_{a'bc'd}$, $\alpha_3(y_{abcd})=y_{a'bcd'}$ and
$\beta_i(y_{abcd})=-y_{abcd}$, for all $1\leq i \leq 3$.
We summarize this in Table~\ref{table: The large group action}.

\begin{table}[h]\caption{The $(\ZZ/2)^6$-action.}\label{table: The large group action}
\newcolumntype{x}[1]{>{\hspace{0pt}}p{#1}}
\centering \renewcommand{\arraystretch}{1.5}
\begin{tabular}[h]{x{.85cm}x{2cm}x{2cm}x{2.4cm}}
\noalign{\hrule height 1pt} %
$\alpha_1$ & $x_{00}\leftrightarrow x_{01}$ & $x_{10}\leftrightarrow x_{11}$ & $y_{abcd}\leftrightarrow y_{a'b'cd}$ \tabularnewline\noalign{\hrule height 0.25pt} 
$\alpha_2$ & $x_{00}\leftrightarrow x_{01}$ & $x_{20}\leftrightarrow x_{21}$ & $y_{abcd}\leftrightarrow y_{a'bc'd}$  \tabularnewline\noalign{\hrule height 0.25pt}
$\alpha_3$ & $x_{00}\leftrightarrow x_{01}$ & $x_{30}\leftrightarrow x_{31}$ & $y_{abcd}\leftrightarrow y_{a'bcd'}$ \tabularnewline\noalign{\hrule height 0.25pt}
$\beta_i$ & $x_{i0}\rightarrow -x_{i0}$ & $x_{i1}\rightarrow -x_{i1}$ & $y_{abcd}\rightarrow -y_{abcd}$ \tabularnewline\noalign{\hrule height 1pt}
\end{tabular}
\end{table}

\medskip

\noindent
Consider the subgroup $G\subset (\ZZ/2)^6$ given by
\begin{equation}\label{eq: definition of small G}
G=\Span{\alpha_1\beta_2,\alpha_2\beta_3,\alpha_3\beta_1}\cong (\ZZ/2)^3.
\end{equation}
It is easy to see that the representation of $G$ on $H^0(\OO_V(1))$ is the regular representation minus the trivial rank $1$ representation; indeed the representation of $G$ on $\Span{x_{00},x_{01},\dots,x_{31}}$ is the regular representation and $x_{00}+x_{01}$ generates the invariant eigenspace. Likewise, given a character $\epsilon \in \Hom((\ZZ/2)^3,\CC)$, it is not hard to see that the polynomial
\begin{equation}\label{eq: eigenvectors of y space}
\sum_{abcd\in \LL} \epsilon(b,c,d) y_{abcd}
\end{equation}
is an eigenvector for the action of $G$ on the space $\Span{y_{abcd}\mid (a,b,c,d)\in\LL}$ and
that the $8$ polynomials obtained in this way generate distinct eigenspaces of the action.
The expression for the trivial eigenvector, obtained from (\ref{eq: eigenvectors of y space})
using the character given by $\epsilon(b,c,d)=(-1)^{b+c+d}$, for all $(b,c,d)\in (\ZZ/2)^3$, is given by:
\begin{equation}\label{eq: the invariant eigenvector of y space}
{\sum_{abcd\in \LL} (-1)^{b+c+d}y_{abcd}}={\sum_{abcd\in \LL} (-1)^a y_{abcd}}.
\end{equation}
\smallskip

\noindent
The representation theory of $G$ on the cohomology of $T$ dictates the eigenspace of $H^0(\OO_V(2))$ from which to take the equation of $T\in \ls{\OO_V(2)}$. According to (\ref{eq: L194}) and the discussion above, the equation for $T$ belongs to the invariant eigenspace of $H^0(\OO_V(2))$.
Consider the following invariant quadratic forms in the $x_{ia}$ variables:
\begin{equation}\label{eq: invariant quadratics}
s_i = \frac{x_{i0}^2+x_{i1}^2}2 \quad \text{and}\quad t_i=x_{i0}x_{i1}\quad
\text{for}\quad i=0,1,2,3.
\end{equation}
using $x_{00}+x_{01}=0$ and (\ref{eq: quadrics}) we obtain
\begin{equation}\label{eq: invariant quadratics companion}
t_i=t_0=-s_0
\end{equation}
on $V$ and $W$. Hence $s_0,s_1,s_2,s_3$ form a basis for the invariant subspace of the second symmetric power of
$H^0(\OO_V(1))$. From this and (\ref{eq: the invariant eigenvector of y space}) we see that a
general element of the invariant eigenspace of $H^0(\OO_V(2))$ is given by:
\begin{equation}\label{eq: quad sec}
q = l+\nu_4 {\sum_{abcd\in \LL} (-1)^a y_{abcd}},\quad \text{where}\quad l=\nu_0s_0 +\nu_1s_1 + \nu_2 s_2 + \nu_3s_3
\end{equation}
and $\nu_0$, $\nu_1$, $\nu_2$, $\nu_3$, $\nu_4$ are general complex parameters.
Let $\mathcal{N}\cong \PP^4$ be the linear system of surfaces given by
\begin{equation}\label{eq: definition of N}
\mathcal{N}=\set{T=V\cap(q=0)\mid (\nu_0,\nu_1,\nu_2,\nu_3,\nu_4)\in
 \PP^4}.
\end{equation}
Then $G$ acts on every $T\in\mathcal{N}$, and we can take the quotient $S=T/G$.

\begin{theorem}\label{thm: free action on $T$}
A general element $T\in\mathcal{N}$ is smooth surface of general type with ample canonical divisor and with $p_g(T)=7$,
$q(T)=0$ and $K_T^2=24$. Furthermore the canonical map of $T$ is a birational morphism
onto a complete intersection of three quadrics and a cubic in $\PP^6$. For a general surface $T\in\mathcal{N}$, the action of $G$ is free
and therefore $S:=T/G$ is a surface of general type with ample canonical divisor and with $p_g(S)=0$ and $K_S^2=3$.
\end{theorem}

\begin{proof}
The base locus of $\mathcal{N}$ is contained in the locus given by $(s_0=s_1=s_2=s_3=0)$.
Using (\ref{eq: invariant quadratics}) and (\ref{eq: invariant quadratics companion}),
we get $x_{i0}x_{i1}=0$ for all $i=0,1,2,3$; and since $2s_i=x_{i0}^2+x_{i1}^2$ we deduce $x_{i0}=x_{i1}=0$
for all $i=0,1,2,3$. Therefore
\[
(s_0=s_1=s_2=s_3=0) \cap V=\set{\yy_{abcd}\mid (a,b,c,d)\in \LL},
\]
which, for general $\nu_0,\nu_1,\nu_2,\nu_3$, does not intersect $T$.
By Bertini's Theorem, $\Sing(T)$ is contained in the union of the base
locus of $\mathcal{N}$ and $\Sing(V)$. For a general choice of $\nu_0,\nu_1,\nu_2,\nu_3,\nu_4$,
the surface $T$ does not meet $\Sing(V)$, (\emph{cf}.~Proposition~\ref{prop: singularities of V}), and as we showed,
$\mathcal{N}$ is base point free. Hence $T$ is nonsingular.
Since the coordinate ring of $T$ is the quotient of $\Run$ by a regular sequence,
it is a Gorenstein graded ring and, in particular, Cohen-Macaulay.
By \cite[Theorem~18.15]{Ei} the coordinate ring of $T$ is a domain and, accordingly, $T$ is reduced and
irreducible. By adjunction, $K_T=\OO_T(1)$ which is ample, and the projectively Gorensteinness of $T$ yields
$q=\dim H^1(K_T)=0$ and $p_g(T)=7$. Finally $K^2_T=\deg(T)=2\deg(V)=24$.

\medskip
\noindent
The canonical map $\varphi_{K_T}$ of $T$ equals $\pi_{|T}$, the map given by the sections
$x_{00}=x_{01}$, $x_{10}$, $x_{11}$, $x_{20}$, $x_{21}$, $x_{30}$, $x_{31}$,
\emph{cf}.~Notation~\ref{not: pi and y}.
Since the locus of common zeros of these sections is contained in the locus
$(s_0=s_1=s_2=s_3=0)$ we deduce that $\varphi_{K_T}$ is a morphism.
Moreover since $K_T$ is ample, $\varphi_{K_T}$ is finite.
Since $\pi_{|V}$ is birational, and $T\in \mathcal{N}$ is a general element
of a movable linear system, $\varphi_{K_T}$ is also birational.
Then the canonical image $\varphi_{K_T}(T)$ is a nondegenerate surface of degree $K_T^2=24$
in the hyperplane  $\PP^6:=(x_{00}+x_{01}=0) \subset \PP^7$,
contained in the locus defined by (\ref{eq: quadrics}). By elimination, we find
a new cubic hypersurface through  $\varphi_{K_T}(T)$.
From $q=0\iff \nu_4{\sum_{abcd\in \LL} (-1)^a y_{abcd}} = -l$,
substituting $y_{abcd}$ with $\frac{x_{1b'}x_{2c'}x_{3d'}}{x_{0a}}$ and using $x_{01}=-x_{00}$
we get
\[
\textstyle  \nu_4\sum_{abcd\in \LL}(-1)^a \frac{x_{1b'}x_{2c'}x_{3d'}}{x_{0a}}= l
\iff \nu_4\sum_{bcd\in\set{0,1}^3} \frac{x_{1b'}x_{2c'}x_{3d'}}{x_{00}} = l,
\]
which yields the irreducible cubic equation:
\begin{equation}\label{eq: cubic}
\nu_4(x_{10}+x_{11})(x_{20}+x_{21})(x_{30}+x_{31})=x_{00}l.
\end{equation}
Therefore $\varphi_{K_T}(T)$ is a surface of degree $24$ contained in the intersection of
the hyperplane $(x_{00}+x_{01}=0)$, the quadrics (\ref{eq: quadrics}) and the
cubic defined by (\ref{eq: cubic}).
Since these polynomials form a regular sequence, we deduce that $\varphi_{K_T}(T)$
\mbox{coincides} with the complete intersection of $3$ quadrics and $1$ cubic that, choosing
$x_{00}$, $x_{10}$, $x_{11}, \dots,x_{30},x_{31}$ as basis for $H^0(K_T)$, are obtained
substituting $x_{01}$ for $-x_{00}$ in (\ref{eq: quadrics}) and (\ref{eq: cubic}).

\medskip
\noindent
Let us now show that the action of $G$ on $T$ is free. By symmetry it is enough to check that
the $3$ elements $\alpha_1\beta_2$, $\alpha_1\alpha_2\beta_2\beta_3$ and $\alpha_1\alpha_2\alpha_3\beta_1\beta_2\beta_3$
act on $T$ without fixed points. In the weighted projective space $\PP(1^8,2^8)$ the fixed locus of an
involution splits into three spaces; the $(+,+)$ part (\emph{i.e.}, positive on the $x$ variables and
positive on the $y$ variables), the $(-,+)$ part and the $(0,-)$ part (\emph{i.e.}, negative on the $y$
variables with all the $x$ variables $0$); since the last space cuts out the empty set on $T$, we will
repeatedly ignore it. Denote these spaces by $\Fix_{(+,+)}$ and $\Fix_{(-,+)}$.
Then, referring to Table~\ref{table: The large group action},
we see that $\Fix_{(+,+)}(\alpha_1\beta_2)$ is equal to:
\[
(x_{00}-x_{01}=x_{10}-x_{11}=x_{20}=x_{21}= y_{abcd}+y_{a'b'cd}=0,\forall_{abcd \in \LL}).
\]
From (\ref{eq: quadrics}) we get $x_{00}x_{01}=x_{10}x_{11}=0$ and hence $x_{00}=x_{01}=x_{10}=x_{11}=0$.
Thus all coordinates $x_{ia}$ vanish except for, possibly, $x_{30}$ or $x_{31}$.
From the quartic relation $y_{abcd}y_{a'b'cd}=x_{2c'}^2x_{3d'}^2=0$,  \emph{cf}.~(\ref{eq: quartic relations}),
and $y_{abcd}+y_{a'b'cd}=0$ we deduce that $y_{abcd}=0$ for all $(a,b,c,d)\in\LL$. Using $q=0$ we obtain $x_{30}=x_{31}=0$.
Hence $T$ does not meet $\Fix_{(+,+)}(\alpha_1\beta_2)$.

\medskip
\noindent
Next we consider the loci $\Fix_{(-,+)}(\alpha_1\beta_2)$, $\Fix_{(+,+)}(\alpha_1\alpha_2\beta_2\beta_3)$, $\Fix_{(-,+)}(\alpha_1\alpha_2\beta_2\beta_3)$ and $\Fix_{(+,+)}(\alpha_1\alpha_2\alpha_3\beta_1\beta_2\beta_3)$ which are given by:
\[
\renewcommand{\arraystretch}{1.3}
\begin{array}{c}
(x_{00}+x_{01}=x_{10}+x_{11}=x_{30}=x_{31}=y_{abcd}+y_{a'b'cd}=0,\forall_{abcd \in \LL}), \\
(x_{10}-x_{11}=x_{20}+x_{21}=x_{30}=x_{31}=y_{abcd}-y_{ab'c'd}=0,\forall_{abcd \in \LL}), \\
(x_{00}=x_{01}=x_{10}+x_{11}=x_{20}-x_{21}=y_{abcd}-y_{ab'c'd}=0,\forall_{abcd \in \LL}), \\
(x_{00}-x_{01}=x_{10}+x_{11}=x_{20}+x_{21}=x_{30}+x_{31}=y_{abcd}+y_{a'b'c'd'}=0,\forall_{abcd \in \LL}),
\end{array}
\]
respectively. Arguing as before (remembering, for the last locus, that $x_{00}+x_{01}=0$ holds)
we see that none of them meets $T$.

\medskip
\noindent
Finally $\Fix_{(-,+)}(\alpha_1\alpha_2\alpha_3\beta_1\beta_2\beta_3)$ is given by:
\[
(x_{00}+x_{01}=x_{10}-x_{11}=x_{20}-x_{21}=x_{30}-x_{31}= y_{abcd}+y_{a'b'c'd'}=0,\forall_{abcd \in \LL}).
\]
 Using (\ref{eq: quadrics}) we get $x_{j0}^2=- x_{00}^2$. Hence $s_j=-s_0$, for all $j=1,2,3$.
From the quartic equations (\ref{eq: quartic relations}) we get
$-y^2_{abcd}=y_{abcd}y_{a'b'c'd'}=x_{20}x_{21}x_{30}x_{31} = x_{00}^4$.
Taking square roots of this equation, substituting in $q=0$ and using the generality of
$\nu_0$, $\nu_1$, $\nu_2$, $\nu_3$, $\nu_4$ we deduce that $x_{00}=0$; and hence
$x_{01}=x_{j0}=x_{j1}=0$ for all $j=1,2,3$ and $y_{abcd}=0$ for all $(a,b,c,d) \in \LL$. Therefore
$T \cap \Fix_{(-,+)}(\alpha_1\alpha_2\alpha_3\beta_1\beta_2\beta_3)=\emptyset$.
\smallskip

\noindent
Since the action of $G$ on $T$ is free, $S=T/G$ is a nonsingular surface of general type with $p_g(S)=0$ and $K^2_S=3$.
Since $K_T$ is ample, we deduce that $K_S$ is ample.
\end{proof}

\begin{remark}
Theorem~\ref{thm: free action on $T$} shows that for every $T\in\N$ such that
\begin{itemize}
\item $T$ has at most canonical singularities,
\item the action of $G$ on $T$ is free,
\end{itemize}
the quotient $S=T/G$ is the canonical model of a surface of general type
with $p_g(S)=0$ and $K^2=3$: this provides a $4$-dimensional family of these surfaces.
\end{remark}

\begin{remark}\label{rmk: needed for Burniat pencil}
By analysis of the proof of Theorem~\ref{thm: free action on $T$}, we see that if, for a given $T\in \mathcal{N}$,
the action of $G$ has any fixed points on $T$ then either $\nu_1\nu_2\nu_3=0$ or there exists $\delta$ in a
finite set of (integer) multiples of $i$ such that $\nu_0-\nu_1-\nu_2-\nu_3+\delta\nu_4 = 0$.
We shall use this observation later on.
\end{remark}

\section{A double cover}\label{sec: double cover}

Consider the Fano $4$-fold $\PP^{1}\times \PP^1 \times \PP^1 \times \PP^1$ with coordinates
$(t_{00},t_{01})$, $(t_{10},t_{11})$, $(t_{20},t_{21})$,
$(t_{30},t_{31})$, and let $\sigma \colon \PP^1\times \PP^1\times \PP^1 \times \PP^1 \rightarrow \PP(1^8,2^8)$ be the map
given by:
\[
\renewcommand{\arraystretch}{1.3}
\begin{array}{c}
\begin{array}{ll}
\sigma^\sharp(x_{0a})=t_{0a'}t_{1a}t_{2a}t_{3a},&\sigma^\sharp(x_{1a})=t_{0a}t_{1a'}t_{2a}t_{3a},  \\
\sigma^\sharp(x_{2a})=t_{0a}t_{1a}t_{2a'}t_{3a}, & \sigma^\sharp(x_{3a})=t_{0a}t_{1a}t_{2a}t_{3a'},
\end{array}\\
\sigma^\sharp(y_{abcd})=t_{0a'}^2t_{1b'}^2t_{2c'}^2t_{3d'}^2,  \text{ if $a=b=c=d$ and }
\sigma^\sharp(y_{abcd})=t_{0a}^2t_{1b}^2t_{2c}^2t_{3d}^2 \text{ otherwise.}
\end{array}
\]
It is straightforward to check that $\sigma(\PP^1\times \PP^1 \times \PP^1 \times \PP^1)= Y\subset \PP(1^8,2^8)$.

\begin{prop}\label{prop: double cover}
The map $\sigma \colon \PP^1\times \PP^1\times \PP^1\times \PP^1 \rightarrow Y$ is finite of degree $2$ branched exactly at the set
$\set{\xx_{ia},\yy_{abcd} : 0\leq i\leq 3, (a,b,c,d)\in \LL}$.
\end{prop}

\begin{proof}
Let $U_{ia}\subset Y$ be the open subset of $Y$ given by $x_{ia}\not = 0$.
First note that $\sigma^{-1}(\yy_{abcd})$ consists of a point, more precisely one of the coordinate points of
$\PP^1\times \PP^1\times \PP^1\times \PP^1$.
Moreover, the family $\set{U_{ia}}$, with $0\leq i\leq 3$ and $a\in \set{0,1}$ is an open affine cover of
$Y \setminus \{\yy_{abcd}\}$.
Consider the restriction $\sigma_{|} \colon \sigma^{-1}(U_{00}) \rt U_{00}$. The open set $\sigma^{-1}(U_{00})$ is
simply $\CC^4$ with coordinates
\[
\frac{t_{00}}{t_{01}},\frac{t_{11}}{t_{10}},\frac{t_{21}}{t_{20}},\frac{t_{31}}{t_{30}}.
\]
The coordinate ring of $U_{00}$, which we denote by $\CC[U_{00}]$, is generated by the regular functions:
\[
\frac{x_{ia}}{x_{00}},\quad \frac{y_{abcd}}{x_{00}^2},\quad  \text{with $0\leq i\leq 3$ and $(a,b,c,d)\in \LL$}.
\]
Computing the image by $\sigma_|^\sharp$ of each of the generators of $\CC[U_{00}]$, we get the generators of the ideal
$(\frac{t_{00}}{t_{01}},\frac{t_{11}}{t_{10}},\frac{t_{21}}{t_{20}},\frac{t_{31}}{t_{30}})^2$. Hence
$\sigma_| \colon \sigma^{-1}(U_{00}) \rt U_{00}$ is finite of degree 2. The same computation on each $U_{ia}$
yields the same result, showing that $\sigma$  is a double cover.
Additionally, the involution $s \in \operatorname{Aut}\pare{\PP^1\times\PP^1 \times \PP^1 \times \PP^1}$ given by
\begin{equation}\label{eq: def of the involution s}
s(t_{ia}) = (-1)^a t_{ia}, \quad \text{for $0\leq i\leq 3$ and $a\in \set{0,1}$}
\end{equation}
satisfies $\sigma \circ s = \sigma$. Note that $s$ has exactly $16$ fixed points, the coordinate
points of $\PP^1\times \PP^1 \times \PP^1\times \PP^1$. In particular, $\sigma$ branches exactly
at their images, \emph{i.e}.~the points in the set $\set{\xx_{ia},\yy_{abcd} : 0\leq i\leq 3, (a,b,c,d)\in \LL}$.
\end{proof}

\begin{remark}
We can deduce from Proposition~\ref{prop: double cover}
that $\Sing Y$ is the set of $16$ points
$\set{\xx_{ia},\yy_{abcd} : 0\leq i\leq 3, (a,b,c,d)\in \LL}$, which are quotient singularities
of type $\frac12(1,1,1,1)$. This agrees with Proposition~\ref{prop: singularities of V}.
\end{remark}

\begin{remark}
The restriction of $\sigma$ to the Fano $3$-fold
\begin{equation}
Z_1=(t_{01}t_{10}t_{20}t_{30}+t_{00}t_{11}t_{21}t_{31}=0) \subset \PP^1\times \PP^1\times \PP^1 \times \PP^1
\end{equation}
is a double cover of $V$, branched on the $14$ singularities of $V$. The $3$-fold $Z_1$ is a (special) member of
$\ls{\OO(1,1,1,1)}^-$, the linear system of  effective divisors on
$\PP^1\times \PP^1\times \PP^1 \times \PP^1$ of degree $(1,1,1,1)$ anti-invariant respect to the involution $s$.
A general member of $\ls{\OO(1,1,1,1)}^-$ is the canonical double cover
of an Enriques--Fano $3$-fold with only terminal singularities. These $3$-folds were classified by Bayle and Sano \cite{bay, sano}.
The image of a general member of $\ls{\OO(1,1,1,1)}^-$ under $\sigma$ falls in case $10$ of Sano's list.
Indeed the whole construction in this section has been inspired by that case.
\end{remark}

Recall that $(\ZZ/2)^6$ acts on $Y$ as given in Table~\ref{table: The large group action}.

\begin{table}[h]\caption{Automorphisms of $\PP^1\times \PP^1\times \PP^1 \times \PP^1$. (For the last $4$,
since the action is diagonal we list only the eigenvalues. Here $\epsilon$ is a square-root of $-1$.)}
\label{table: automorphisms of P1xP1xP1xP1}
\newcolumntype{x}[1]{>{\raggedleft\hspace{0pt}}p{#1}}
\centering \renewcommand{\arraystretch}{1.4}
\begin{tabular}[h]{x{.75cm}x{.75cm}x{.75cm}x{.75cm}x{.75cm}x{.75cm}x{.75cm}x{.75cm}x{.75cm}x{.01cm}}
\noalign{\hrule height 1pt} %
           & $t_{00}$ & $t_{01}$ & $t_{10}$ & $t_{11}$ & $t_{20}$ & $t_{21}$ & $t_{30}$ & $t_{31}$ & \tabularnewline\noalign{\hrule height 0.25pt} 
$\tilde{\alpha}_1$ &  $t_{10}$ & $t_{11}$ & $t_{00}$ & $t_{01}$ & $t_{31}$ & $t_{30}$ & $t_{21}$ & $t_{20}$ & \tabularnewline\noalign{\hrule height 0.25pt} 
$\tilde{\alpha}_2$ & $t_{20}$ & $t_{21}$ & $t_{31}$ & $t_{30}$ & $t_{00}$ & $t_{01}$ & $t_{11}$ & $t_{10}$ & \tabularnewline\noalign{\hrule height 0.25pt} 
$\tilde{\alpha}_3$ & $t_{30}$ & $t_{31}$ & $t_{21}$ & $t_{20}$ & $t_{11}$ & $t_{10}$ & $t_{00}$ & $t_{01}$ & \tabularnewline\noalign{\hrule height 0.25pt} 
$\tilde{\beta}_1$  & -$\epsilon$ & $1$ & $1$ & $-\epsilon$ & $1$ & $\epsilon$ & $1$ & $\epsilon$ & \tabularnewline\noalign{\hrule height 0.25pt}  
$\tilde{\beta}_2$  & -$\epsilon$ & $1$ & $1$ & $\epsilon$ & $1$ & $-\epsilon$ & $1$ & $\epsilon$ & \tabularnewline\noalign{\hrule height 0.25pt} 
$\tilde{\beta}_3$  & -$\epsilon$ & $1$ & $1$ & $\epsilon$ & $1$ & $\epsilon$ & $1$ &  $-\epsilon$ &  \tabularnewline\noalign{\hrule height 0.25pt} 
$s$ & $1$ & $-1$ & $1$ & $-1$ & $1$ & $-1$ & $1$ & $-1$  \tabularnewline\noalign{\hrule height 1pt} 
\end{tabular}
\end{table}
In Table~\ref{table: automorphisms of P1xP1xP1xP1}, we distinguish a set of automorphisms of
$\PP^1\times \PP^1\times \PP^1\times \PP^1$, one of which, $s$, has already been defined in (\ref{eq: def of the involution s})
and the remaining ones are meant to lift the actions of $\alpha_1,\alpha_2,\alpha_3,\beta_1,\beta_2,\beta_3$.
One can check by direct computation that
$\tilde{\alpha}_1,\tilde{\alpha}_2,\tilde{\alpha}_3,\tilde{\beta}_1,\tilde{\beta}_2,\tilde{\beta}_3$
lift the action of $\alpha_1,\alpha_2,\alpha_3,\beta_1,\beta_2,\beta_3$, i.e., that
$\sigma\circ \tilde{\alpha}_i= \alpha_1\circ \sigma$ and $\sigma\circ \tilde{\beta}_i= \beta_1\circ \sigma$, for $i=1,2,3$.
On the other hand, there are a number of small checks that are straightforward. It is clear that $s$ commutes with every other automorphism listed
in Table~\ref{table: automorphisms of P1xP1xP1xP1}; it is also clear that $\tilde{\alpha}_1,\tilde{\alpha}_2,\tilde{\alpha}_3$ are automorphisms of order $2$ commuting with
each other, that $\tilde{\beta}_1,\tilde{\beta}_2,\tilde{\beta}_3$ commute with each other and that $\tilde{\beta}^2_1=\tilde{\beta}^2_2=\tilde{\beta}^2_3=s$.
Finally, a less straightforward (but still elementary) computation shows that $\tilde{\alpha}_i\tilde{\beta}_j = s^{\delta_{ij}}\tilde{\beta}_j\tilde{\alpha}_i$, where $\delta_{ij}$ is Kronecker's delta.
These identities are useful in the proof of the next proposition, where we characterize the group $\tilde{G}$ generated by the automorphisms
that lift the generators of $G=\Span{\alpha_1\beta_2,\alpha_2\beta_3,\alpha_3\beta_1}\simeq (\ZZ/2)^3$.

\begin{lemma}\label{lemma: lifting G}
$\tilde{G}:=\langle \tilde{\alpha}_1\tilde{\beta}_2, \tilde{\alpha}_2\tilde{\beta}_3, \tilde{\alpha}_3\tilde{\beta}_1 \rangle$
is isomorphic to $\ZZ/2 \times Q_8$, where $Q_8$ is the classical quaternion group.
\end{lemma}

\begin{proof}
Since $\deg \sigma=2$, $|\tilde{G}|$ equals either $2|G|$, if $s \in \tilde{G}$, or $|G|$, if  $s \not\in \tilde{G}$.
Since $(\tilde{\alpha_1}\tilde{\beta_2})^2=s$, we get $|\tilde{G}|= 2\ls{G}=16$. Consider
the standard presentation of $Q_8$ given by
\[
\langle -1,i,j,k\mid (-1)^2=1,i^2=j^2=k^2=ijk=-1\rangle
\]
and, for clarity, let us use multiplicative notation for $\ZZ/2=\set{1,-1}$.
Set:
\[\renewcommand{\arraystretch}{1.3}
\begin{array}{c}
\mu(1,-1)=s,\quad
\mu(-1,1)=\tilde{\alpha}_1\tilde{\beta}_2 \tilde{\alpha}_2\tilde{\beta}_3 \tilde{\alpha}_3\tilde{\beta}_1, \\
\mu(1,i)=\tilde{\alpha}_2\tilde{\beta}_3 \tilde{\alpha}_3\tilde{\beta}_1,\quad
\mu(1,j)=\tilde{\alpha}_3\tilde{\beta}_1 \tilde{\alpha}_1\tilde{\beta}_2,\quad
\mu(1,k)=\tilde{\alpha}_1\tilde{\beta}_2 \tilde{\alpha}_2\tilde{\beta}_3.
\end{array}
\]
Using the identities stated earlier, one can check easily that these definitions respect all the relations of
$(\ZZ/2)\times Q_8$ and therefore determine a group homomorphism $\mu \colon (\ZZ/2)\times Q_8 \rt \tilde{G}$.
Since:
\[\renewcommand{\arraystretch}{1.3}
\begin{array}{c}
\mu(-1,-i)=\mu(-1,1)\mu(1,i)^{-1} = \tilde{\alpha}_1\tilde{\beta}_2 \tilde{\alpha}_2\tilde{\beta}_3 \tilde{\alpha}_3\tilde{\beta}_1
\tilde{\beta}_1^{-1}\tilde{\alpha}_3\tilde{\beta}_3^{-1}\tilde{\alpha}_2 = \tilde{\alpha}_1\tilde{\beta}_2, \\
\mu(-1,-j)= (\tilde{\alpha}_1\tilde{\beta}_2) \tilde{\alpha}_2\tilde{\beta}_3 \tilde{\alpha}_3\tilde{\beta}_1 (
\tilde{\beta}_2^{-1}\tilde{\alpha}_1) \tilde{\beta}_1^{-1}\tilde{\alpha}_3 =
\tilde{\alpha}_2\tilde{\beta}_3 \tilde{\alpha}_3\tilde{\beta}_1 \tilde{\beta}^{-1}_1\tilde{\alpha}_3 =
\tilde{\alpha}_2\tilde{\beta}_3, \\
\mu(-1,-k)= \tilde{\alpha}_1\tilde{\beta}_2 (\tilde{\alpha}_2\tilde{\beta}_3) \tilde{\alpha}_3\tilde{\beta}_1
(\tilde{\beta}_3^{-1}\tilde{\alpha}_2)\tilde{\beta}_2^{-1}\tilde{\alpha}_1 = s
\tilde{\alpha}_1\tilde{\beta}_2 \tilde{\alpha}_3\tilde{\beta}_1\tilde{\beta}^{-1}_2\tilde{\alpha}_1 =
\tilde{\alpha}_3\tilde{\beta}_1,
\end{array}
\]
we deduce that $\mu$ is surjective, which, as $|\tilde{G}|=|(\ZZ/2)\times Q_8|$, implies that $\mu$ is an isomorphism.
\end{proof}

We can now give a good description of the family of surfaces $T/G$, for general $T$ in the linear system ${\mathcal N}$.

\begin{theorem}\label{thm: universal cover of $T$}
Let $T\in\mathcal{N}$ be a surface with at most canonical singularities for which the action of $G$ on it is free. Then
$\pi_1(T/G) \cong \ZZ/2 \times Q_8$ and the universal cover of $T$ is a complete intersection of the two hypersurfaces
in $\PP^1\times \PP^1\times \PP^1\times \PP^1$,  $Z_1$ and $Z_2$, of multi-degrees $(1,1,1,1)$ and $(2,2,2,2)$, respectively,
given by:
\[
Z_1=(t_{01}t_{10}t_{20}t_{30}+t_{00}t_{11}t_{21}t_{31}=0)\text{ and }
\]
\[
Z_2=\sum_{i=0}^3 \nu_i \left(t_{i0}^2\prod_{j\neq i}t_{j1}^2+t_{i1}^2\prod_{j\neq i}t_{j0}^2\right) - 2\nu_4 {\sum_{abcd\in \LL} (-1)^{\frac{b+c+d-a}2} t_{0a}^2t_{1b}^2t_{2c}^2t_{3d}^2} =0.
\]
\end{theorem}

\begin{proof}
We note that $T$ does not contain any of the $16$ points in the set
\[
\set{\xx_{ia},\yy_{abcd}\mid 0\leq i\leq 3, (a,b,c,d)\in \LL}.
\]
Indeed $T$ is a Cartier divisor in $V$, which contains $14$ of these points that,
by Proposition \ref{prop: singularities of V}, are singular points of $V$ with Zariski tangent
space of dimension $5$ or $8$. In particular, if $T$ contains one of these points, the Zariski tangent
space of $T$ at this point has at least dimension $4$, whereas every canonical singularity of a surface
has Zariski tangent space of dimension $3$. Since $T$ is the complete intersection of two divisors ($V$ and a quadric section, given by $x_{00}+x_{01}=0$ and
(\ref{eq: quad sec}), respectively) in $Y$, the surface $\tilde{T}:=\sigma^{-1}(T)$ is the complete intersection
of their pull-back to $\PP^1\times \PP^1 \times \PP^1 \times \PP^1$, which one easily sees are the hypersurfaces
$Z_1$ and $Z_2$, respectively, of the statement of this theorem. By the Leftschetz hyperplane section theorem, $\pi_1(\tilde{T})=0$.
Now as the composition $\tilde{T} \stackrel{\sigma_|}{\rightarrow} T \rightarrow S$ is \'{e}tale, we conclude that
$\tilde{T}$ is the universal cover of $S$. In particular, $\pi_1(S)$ is isomorphic to the group of automorphisms of the cover,
which coincides with the group of automorphisms of $\tilde{T}$ lifting the action of $G$. This is $\tilde{G}$, which, by Lemma \ref{lemma: lifting G},
is isomorphic to $\ZZ/2 \times Q_8$.
\end{proof}

We conclude this section by studying the locus of the moduli space of the surfaces of general type described by the surfaces $S$.

\begin{theorem}\label{thm: moduli}
Let $\mathcal{U}$ be the dense open set of $\mathcal{N}\cong\PP^4$ consisting of the surfaces $T$ with at most canonical singularities
on which $G$ acts freely. Then, the map associating to each point of $\mathcal{U}$ the class of the surface $S/G$,
in the Gieseker moduli space of surfaces of general type with $\chi=1$ and $K^2=3$, is finite.
In particular, its image is $4$-dimensional and unirational.
\end{theorem}

\begin{proof}

\noindent
Let $S_1:=T_1/G$, $S_2:=T_2/G$ be surfaces with $T_1,T_2\in \mathcal{U}$. Assume that \mbox{$S_1 \cong S_2$}.
By Theorem~\ref{thm: universal cover of $T$},
$\pi_1(S_1)\cong \pi_1(S_2)\cong (\ZZ/2)\times Q_8$. Since the Abelianization of
$(\ZZ/2)\times Q_8$ is $(\ZZ/2)^3$,
each $S_i$ has exactly one $(\ZZ/2)^3$-cover up to isomorphism.
Therefore from $S_1 \cong S_2$, it follows that $T_1 \cong T_2$.
This isomorphism induces an isomorphism of the canonical rings of $T_1$ and $T_2$. Choose an
automorphism $\Phi$ of $\PP(1^8,2^8)$ that lifts the isomorphism between $\Proj R(T_1,K_{T_1})$ and $\Proj R(T_2,K_{T_2})$.
Note that $\Phi$ is not unique, as the image by $\Phi^\sharp$ of each ge\-ne\-ra\-tor of the underlying polynomial ring
is determined only modulo the ideal of $T_2$ and therefore, in particular, $\Phi^\sharp(x_{ia})$ is determined only up to
$x_{00}+x_{10}$. In what follows we show that $\Phi_{|(x_{00}+x_{01}=0)}$ belongs to a finite set.
\medskip

\noindent
The restriction of the isomorphism $\Phi^\sharp$ to the variables of degree $1$
yields an automorphism of $\PP^7$, which we denote by $\tilde{\Phi}$, mapping the canonical image of $T_1$ to
the canonical image of $T_2$. In particular $\tilde{\Phi}$ preserves the hyperplane $(x_{00}+x_{01}=0)$, which is the linear
span of both surfaces, and $W\subset \PP^7$, given by (\ref{eq: quadrics}), which is their quadric hull.
Recall that, as in Definition~\ref{def: the unprojection and projection maps}, for every $(a,b,c,d)\in \LL$, $H_{abcd}$ is the divisor of poles
of the rational function $\varphi^\sharp(y_{abcd})=\varphi_{abcd}$ on $W$.
Let us consider the $8$ planes given by
$\Pi_{abcd}:=H_{abcd} \cap (x_{00}+x_{01}=0) \subset W$.
Then $\tilde{\Phi}^{-1}(\Pi_{abcd})$ is a plane, it is contained in $W$,
and it is the intersection of $(x_{00}+x_{01}=0)$ with the divisor of the
poles of $\varphi^{\sharp}\Phi^{\sharp}(y_{abcd})$. Since $\Phi^{\sharp}(y_{abcd}) \in \operatorname{S}^2\Span{x_{ia}}\oplus \langle y_{abcd} \rangle$,
we deduce that $\tilde{\Phi}^{-1}$ permutes the $8$ planes $\Pi_{abcd}$.
The first consequence is that $\tilde{\Phi}^{-1}$ preserves the linear span of these $8$ planes,
$(x_{00}=x_{10}=0)$. Now, as we are only interested on $\Phi_{|(x_{00}+x_{01}=0)}$, we may modify
$\Phi$ so that $\Phi^\sharp(x_{0a})=\lambda x_{0a}$, for $a=0,1$ and for some $\lambda \in \CC^*$. Then, rescaling
$\Phi^\sharp$, and thus still without changing $\Phi_{|(x_{00}+x_{01}=0)}$, we may finally assume that $\Phi^\sharp(x_{0a})= x_{0a}$.
Another consequence of the fact that $\tilde{\Phi}^{-1}$ permutes the $8$ planes $\Pi_{abcd}$
is that there exists $\tau \in {\mathfrak S}_3$, $(a,b,c,d)\in \LL$, with $a=0$
and  $\lambda_1,\lambda_2,\lambda_3\in \CC^*$ such that $\Phi^\sharp(x_{10})=\lambda_1x_{\tau(1)b}$, $\Phi^\sharp(x_{20})=\lambda_2x_{\tau(2)c}$ and
$\Phi^\sharp(x_{30})=\lambda_3x_{\tau(3)b}$. Since $\Phi^\sharp$ must also preserve
(\ref{eq: quadrics}) we deduce that $\Phi^\sharp(x_{11})=\lambda_1^{-1}x_{\tau(1)b'}$, $\Phi^\sharp(x_{21})=\lambda_2^{-1}x_{\tau(2)c'}$ and
$\Phi^\sharp(x_{31})=\lambda_3^{-1}x_{\tau(3)b'}$.
\medskip

\noindent
Consider the action of ${\mathfrak S}_3$ on $\PP(1^8,2^8)$ given, for every $\tau \in {\mathfrak S}_3$ by,
\[
\tau^\sharp(x_{0a})=x_{0a}, \quad
\tau^\sharp (x_{ia}) = x_{\tau(i)a}, \quad \tau^\sharp(y_{a_0a_1a_2a_3})=y_{a_0a_{\tau^{-1}(1)}a_{\tau^{-1}(2)}a_{\tau^{-1}(3)}},
\]
for all $1\leq i \leq 3$ and $(a,b,c,d)\in \LL$.
We note that given $\tau \in \mathfrak{S}_3$, we have
$\tau \alpha_i \tau^{-1}=\alpha_{\tau(i)}$ and $\tau \beta_i \tau^{-1}=\beta_{\tau(i)}$, where, recall,
$\alpha_i$ and $\beta_i$ generate $(\ZZ/2)^6$ and act on $\PP(1^8,2^8)$ as given in Table~\ref{table: The large group action}.
These actions generate a finite group $\Lambda$ of automorphisms of $\PP(1^8,2^8)$ preserving $V$ and $Y$ which is a semidirect product,
$\Lambda \cong (\ZZ/2)^6 \rtimes {\mathfrak S}_3$. Accordingly, going back to $\Phi$, there exists $\Psi \in\Lambda$ and constants $\lambda_{ia}\in\CC^*$ such that
\[
(\Phi \circ \Psi)^{\sharp}(x_{ia})= (\Psi^{\sharp} \circ \Phi^{\sharp} )(x_{ia})=\lambda_{ia}x_{ia}
\]
with $\lambda_{00}= \lambda_{01}=1$ and $\lambda_{i1}=\lambda_{i0}^{-1}$. Notice that from the cubic relations (\ref{eq: cubic relations}) we get
\mbox{$(\Phi \circ \Psi)^{\sharp}(y_{abcd})=\lambda_{1b'}\lambda_{2c'}\lambda_{3d'}y_{abcd}$}.
Since for both $T_1,T_2\in \mathcal{U}\subset \mathcal{N}$ we must have $\nu_4 \neq 0$
(for otherwise $T_1$ or $T_2$ would be too singular), from the equation of the quadric section (\ref{eq: quad sec}), we deduce that
the products $\lambda_{1b'}\lambda_{2c'}\lambda_{3d'}$, for $(a,b,c,d)\in \LL$ are all equal. This can only happen if $\lambda_{i1}=\lambda_{i0}$ for $i\in \set{1,2,3}$. Hence $\lambda_{ia} \in \{\pm 1\}$, for $i=1,2,3$
and thus there are only finitely many possibilities for $\Phi_{|(x_{00}+x_{01})=0} \circ \Psi$. Since $\Psi$ belongs to the finite group $\Lambda$,
we deduce that there are, as well, only finitely many possibilities for $\Phi_{|(x_{00}+x_{01})=0}$.
\end{proof}

\section{The bicanonical map of $S$}\label{sec: bicanonical}
The main goal of this section is to compare the surfaces we have constructed with the other constructions existing in literature.
To reach this goal we study the bicanonical map of $S$, which is interesting in its own right. We show that, as in the Burniat case, the bicanonical map of $S$ is a bidouble cover of a cubic surface in $\PP^3$ with $3$ nodes.
This induces a birational description of these surfaces as bidouble covers of the plane. We compute the branch divisors, and then identify the Burniat surfaces.
Not surprisingly, the branch divisors corresponding to a general surface in our family (cf. Figure \ref{fig: twc})
correspond exactly to the one used in the recent paper \cite{bcburniat3} to define the {\it extended} tertiary Burniat surfaces.
\medskip

\noindent
Consider the action of $(\ZZ/2)^6$ on $V$ as given in Table~\ref{table: The large group action}. For a subgroup of $(\ZZ/2)^6$ to act on $T$ it must preserve the equation $q=0$. An element of $(\ZZ/2)^6$, written as $\alpha_1^{a_1}\alpha_2^{a_2}\alpha_3^{a_3}\beta_1^{b_1}\beta_2^{b_2}\beta_3^{b_3}$, sends $q=0$ to a scalar multiple of it if and only if the integer $a_1+a_2+a_3+b_1+b_2+b_3$ is even. Let $H$ be the subgroup of $(\ZZ/2)^6$ given by
\begin{equation}\label{eq: definition of large H}
H = \set{\alpha_1^{a_1}\alpha_2^{a_2}\alpha_3^{a_3}\beta_1^{b_1}\beta_2^{b_2}\beta_3^{b_3} \in (\ZZ/2)^6 \mid a_1+a_2+a_3+b_1+b_2+b_3
\text{ is even}}.
\end{equation}
The group $G$ defined in (\ref{eq: definition of small G}) is obviously a subgroup of $H$. Hence the quotient $\Gamma:=H/G\cong (\ZZ/2)^2$ acts on $S=T/G$. Denote by $\gamma \colon S \rt S/\Gamma$ the quotient morphism. In the next proposition we show that $\gamma$ is the bicanonical map of $S$.

\begin{prop}\label{prop: the bicanonical map of S}
Let $T\in\mathcal{N}$ be a surface with at most canonical singularities and such that the action of $G$ on it is fixed-point free. Consider $S=T/G$. Then, the
bicanonical map of $S$ is a bidouble cover of the cubic surface $S_3\subset \PP^3$ given by
$8\nu_4^2(s_1-s_0)(s_2-s_0)(s_3-s_0) - s_0 (\nu_0s_0+\nu_1s_1+\nu_2s_2+\nu_3s_3)^2=0$.
\end{prop}
\begin{proof}
The bicanonical system of $S$ is generated by the $4$ invariants quadratic forms $s_0,s_1,s_2,s_3$.
We showed in the proof of Theorem~\ref{thm: free action on $T$} that $s_0=s_1=s_2=s_3=0$ cuts
out the empty set on $T$; therefore $\ls{2K_S}$ has no fixed part and no base points.
Since $S$ is a minimal surface of general type with $p_g=0$ and $K^2\geq 2$, by \cite{X},
the bicanonical system is not composed with a pencil. Hence the image of $\varphi_{2K_S}$ is a surface.
To find its equation, we square both sides of (\ref{eq: cubic}):
\[
\nu_4^2(x_{10}+x_{11})^2(x_{20}+x_{21})^2(x_{30}+x_{31})^2=x_{00}^2l^2
\]
and use $(x_{i0}+x_{i1})^2= 2(s_i +t_i)=2(s_i-s_0)$, for $i=1,2,3$, and $s_0=x_{00}^2$,
\emph{cf}.~(\ref{eq: invariant quadratics}), (\ref{eq: invariant quadratics companion}).
Substituting, we get $8\nu_4^2(s_1-s_0)(s_2-s_0)(s_3-s_0) - s_0 l^2 =0$.
For a general choice of $\nu_0,\nu_1,\nu_2,\nu_3,\nu_4$ this cubic is irreducible, hence the cubic surface $S_3\subset \PP^3$ it defines coincides with $\varphi_{2K_S}(S)$. $S$ has no $(-2)$-curves, as by construction $K_S$ is ample, thus $\varphi_{2K_S}$ is a finite morphism of degree $4$. Since $s_0,s_1,s_2,s_3$ are invariant for the action of $H$ on $T$, $\varphi_{2K_S}$ factors through $\gamma$,
which is also a finite morphism of degree $4$. Hence, since $S_3$ is normal (\emph{cf}.~Remark~\ref{rem: nodes on cubic}), $S/\Gamma \cong S_3$ and, up to isomorphism, $\varphi_{2K_S}=\gamma$.
\end{proof}

\begin{remark}\label{rem: nodes on cubic}
For general $\nu_0,\dots,\nu_4$, the cubic $S_3\subset \PP^3$ has $3$ ordinary double points:
\begin{equation}\label{eq: nodes of S_3}
\begin{array}{c}
n_1 = (s_2-s_0=s_3-s_0=\nu_0s_0+\nu_1s_1+\nu_2s_2+\nu_3s_3=0),\\
n_2 = (s_1-s_0=s_3-s_0=\nu_0s_0+\nu_1s_1+\nu_2s_2+\nu_3s_3=0),\\
n_3 = (s_1-s_0=s_2-s_0=\nu_0s_0+\nu_1s_1+\nu_2s_2+\nu_3s_3=0);\\
\end{array}
\end{equation}
and these are the only singularities of $S_3$.
\end{remark}

Let us denote by $\theta_i\in \Gamma = H/G$ the class of $\alpha_i\beta_i$, \emph{i.e.},
\begin{equation}\label{eq: the theta_i}
\theta_i = [\alpha_i\beta_i] = \set{\alpha_i\beta_ig \mid g\in G}.
\end{equation}
By Proposition~\ref{prop: the bicanonical map of S}, $\varphi_{2K_S}$ is the quotient by the action of
 $\Gamma=\{1,\theta_1,\theta_2,\theta_3\}$.
To study this map, by the general theory of the bidouble covers
(see \cite{C}), we study its branch locus. A ramification point of $\varphi_{2K_S}$ is the image of a point
$\xx \in S$ fixed by some $\theta_i$, \emph{i.e}., for which $I_\xx \not = \set{1}$, where $I_\xx = \set{g\in \Gamma \mid g\xx = \xx}$ is the inertia group of $\xx$. When $S$ is smooth, there are $3$ possibilities for $I_\xx$:
\begin{enumerate} \renewcommand{\labelenumi}{(\alph{enumi})}
\item $I_\xx=\langle \theta_i \rangle$ and $\xx$ is an isolated fixed
point of $\theta_i$. Then, in suitable local coordinates, $\theta_i$ acts by
$(z_1,z_2) \mapsto (-z_1,-z_2)$ and $\varphi_{2K_S}(\xx)$ is a node.
\item $I_\xx=\langle \theta_i \rangle$ and $\xx$ is not an isolated fixed
point of $\theta_i$. Then, in local coordinates, $\theta_i$ acts by
$(z_1,z_2) \mapsto (-z_1,z_2)$ and the locus of all such points is a smooth curve $R_i\subset S$.
\item $I_\xx=\Gamma$. Then $\xx$ belongs to exactly two $R_i$, intersecting transversally in $\xx$.
\end{enumerate}
Let $D_i:=\varphi_{2K_S}(R_i)$ and denote by $\Delta_i$ the image of the set of
isolated fixed points of $\theta_i$ the inertia group of which is not the whole of $\Gamma$ --- as in type (a), above.
Then each $\Delta_i$ is a set of nodes of $S_3$. The bidouble cover is determined by $D_1,D_2,D_3$ and $\Delta_1,\Delta_2,\Delta_3$.
\medskip

\noindent
To describe $D_i$ we introduce some notation.
The intersection of $S_3$ with the plane $s_0+s_i=0$ splits as the union of a
line with a conic. Denote these by $L_i$ and $C_i$, respectively. In other
words, set
\begin{equation}\label{eq: equations of L_i and C_i}
\renewcommand{\arraystretch}{1.3}
\begin{array}{c}
L_i = (s_0=s_i=0) \text{ and }\\
C_i = (s_0+s_i = 16\nu_4^2(s_{i+1}-s_0)(s_{i+2}-s_0) + l^2=0),
\end{array}
\end{equation}
taking the indices in $\set{1,2,3}$, modulo $3$.

\begin{prop}\label{prop: the branch locus of the bicanonical map}
Let $D_i$, $\Delta_i$, for $i=1,2,3$, be the branch loci of the map $\varphi_{2K_S}\colon S \rt
S_3\subset \PP^3$. Then $\Delta_i=\set{n_i}$ and $D_i=C_{i+1} + L_{i-1}$, taking indices in $\set{1,2,3}$, modulo $3$.
\end{prop}

\begin{proof} By cyclic symmetry, it is enough to compute $\Delta_1$ and $D_1$.
On the other hand, the fixed points of $\theta_1$ are the
images on $S$ of the points of $T$ fixed by an element of $[\alpha_1\beta_1]$.
Recall that the elements of $\theta_1=[\alpha_1\beta_1]$ are
$\alpha_1\beta_1$, $\beta_1\beta_2$, $\alpha_1\alpha_2\beta_1\beta_3$, $\alpha_1\alpha_3$, $\alpha_2\beta_1\beta_2\beta_3$,
$\alpha_3\beta_2$, $\alpha_1\alpha_2\alpha_3\beta_3$ and $\alpha_2\alpha_3\beta_2\beta_3$.

\medskip
\noindent
$\Fix_{(+,+)}(\alpha_2\beta_1\beta_2\beta_3)$ and $\Fix_{(-,+)}(\alpha_2\beta_1\beta_2\beta_3)$ are given by:
\[
\begin{array}{c}
(x_{00}-x_{01}=x_{10}=x_{11}=x_{20}+x_{21}=x_{30}=x_{31}=y_{abcd}+y_{a'bc'd}=0,\forall_{abcd \in \LL}), \\
(x_{00}+x_{01}=x_{20}-x_{21}=y_{abcd}+y_{a'bc'd}=0,\forall_{abcd \in \LL}), \\
\end{array}
\]
respectively. We have $\Fix_{(+,+)}(\alpha_2\beta_1\beta_2\beta_3)\cap T = \emptyset$. This can be seen either directly on $T$
or by noticing that its image in $S_3$ must have $s_0=s_1=s_3=0$ and, by (\ref{eq: quadrics}), $s_2=0$.
In $\Fix_{(-,+)}(\alpha_2\beta_1\beta_2\beta_3)\cap T$  we have $x_{00}=-x_{01}$ and $x_{20}=x_{21}$,
which by (\ref{eq: quadrics}) imply that $s_0+s_2=0$. We deduce that the image of $\Fix_{(-,+)}(\alpha_2\beta_1\beta_2\beta_3)\cap T$ in $S_3$ is contained in $L_2\cup C_2$. Suppose that $s_0=s_2=0$. Then
\[x_{00}=x_{01}=x_{20}=x_{21}=x_{10}x_{11}=x_{30}x_{31}=0.
\]
Assume that $x_{10}=x_{30}=0$. Then using (\ref{eq: quartic relations}), we get $y_{abcd}^2 = -y_{abcd}y_{a'bc'd}=0$, for all $(a,b,c,d)\in \LL\setminus\set{(0,0,0,0),(1,0,1,0)}$. Hence we are left with the $2$ equations:
\[
-y_{0000}^2 = x_{11}^2x_{31}^2 = 4 s_1s_3 \quad \text{and} \quad 2\nu_4y_{0000} + l=0,
\]
given by (\ref{eq: quartic relations}) and $q=0$. Eliminating $y_{0000}$, we get the equation of $C_2$ with $s_0=s_2=0$. This is independent of the choices we made. We deduce that the image of $\Fix_{(-,+)}(\alpha_2\beta_1\beta_2\beta_3)\cap T$ is contained in $C_2$. To see that the image of this locus coincides with $C_2$ it suffices to check that it is $1$-dimensional. The equations $x_{00}+x_{01}=0$ and $x_{20}-x_{21}=0$ define in $W$ a $2$-dimensional subscheme (in fact, $x_{00}+x_{01}$ is an equation of $W$).
It is clear that this subscheme is not contained in the exceptional locus of $\varphi \colon W \drt V$. Denote by $Z$ its strict transform in $V$. Assume $x_{00},x_{01}\not = 0$. Then
\[
y_{abcd}x_{0a} =  x_{1b'}x_{2c'}x_{3d'} = x_{1b'}x_{2c}x_{3d'}=-y_{a'bc'd}x_{0a'}\implies y_{abcd}=-y_{a'bc'd}.
\]
Hence on the open set $x_{00}\not = 0$ of $Z\subset V$, the equations
$y_{abcd}+y_{a'bc'd}=0$ are redundant. Hence $\dim Z=2$. Since we obtain $T$
from $V$ by taking a hypersurface section ($q=0$) we deduce that
$\Fix_{(-,+)}(\alpha_2\beta_1\beta_2\beta_3)\cap T$ is $1$-dimensional.
We conclude that the fixed points of $\alpha_2\beta_1\beta_2\beta_3$ do not contribute to $\Delta_1$ and
that their contribution to $D_1$ is $C_2$.
\medskip

\noindent
$\Fix_{(+,+)}(\beta_1\beta_2)$ and $\Fix_{(-,+)}(\beta_1\beta_2)$ are given by:
\begin{equation}
(x_{10}=x_{11}=x_{20}=x_{21}=0)\quad\text{and}\quad (x_{00}=x_{01}=x_{30}=x_{31}=0),
\end{equation}
respectively. The image of $\Fix_{(-,+)}(\beta_1\beta_2)\cap T$ equals
 $L_3=(s_0=s_3=0)$: it is clearly contained in $L_3$ and the equality follows
 since $\Fix_{(-,+)}(\beta_1\beta_2) \supset \SSS^{12}_{11}$ and
 $\SSS^{12}_{11} \cap T$ is positive dimensional. Hence $L_3 \subset D_1$.
Notice that by symmetry of the indices we have just shown that $L_1\subset
 D_2$ and $L_2 \subset D_3$.
For the image of $\Fix_{(+,+)}(\beta_1\beta_2)\cap T$, in $S_3$ we get $s_1=s_2=0$
 and then $s_0=0$, which means that $\Fix_{(+,+)}(\beta_1\beta_2)\cap T$ consists of the
 preimages of the point $L_1\cap L_2$; these are points in $R_2 \cap R_3$ --- type (c) above.
We conclude that the fixed points of $\beta_1\beta_2$ do not contribute to $\Delta_1$ and
that their contribution to $D_1$ is $L_3$.
\medskip

\noindent
$\Fix_{(+,+)}(\alpha_2\alpha_3\beta_2\beta_3)$ and $\Fix_{(-,+)}(\alpha_2\alpha_3\beta_2\beta_3)$ are given by:
\[
\begin{array}{c}
(x_{20}+x_{21}=x_{30}+x_{31}=y_{abcd}-y_{abc'd'}=0,\forall_{abcd \in \LL}), \\
(x_{00}=x_{01}=x_{10}=x_{11}=x_{20}-x_{21}=x_{30}-x_{31}=y_{abcd}-y_{abc'd'}=0,\forall_{abcd \in \LL}), \\
\end{array}
\]
respectively. The locus $\Fix_{(-,+)}(\alpha_2\alpha_3\beta_2\beta_3)\cap T$ is clearly empty. For the image of the locus
$\Fix_{(+,+)}(\alpha_2\alpha_3\beta_2\beta_3)\cap T$ we get $s_0-s_2=0$ and $s_0-s_3=0$ and then from the equation of $S_3$, $s_0l^2=0$. If $s_0=0$ then $s_2=s_3=0$ and then in $\Fix_{(+,+)}(\alpha_2\alpha_3\beta_2\beta_3)$ we get
$x_{00}=x_{01}=x_{20}=x_{21}=x_{30}=x_{31}=x_{10}x_{11}=0$. From this we deduce that all $y_{abcd}$ are zero, which together with $q=0$ forces all variables to be zero. Hence $s_0\not = 0$ and we must have $l=0$. To show that the image of
 $\Fix_{(+,+)}(\alpha_2\alpha_3\beta_2\beta_3)\cap T$ coincides with $n_1$, as in (\ref{eq: nodes of S_3}),
 it suffices to show that this locus is nonempty.  The equations $x_{20}+x_{21}=x_{30}+x_{31}=0$ define in $W$ a subscheme of dimension $1$ which is not contained in the exceptional locus
 of $\varphi_{|W} \colon  W \drt V$, hence in $V$ they define a positive-dimensional subscheme $Z'\subset V$.
If $x_{00},x_{01}\not =0$ then,
\[
y_{abcd}x_{0a} = x_{1b'}x_{2c'}x_{3d'}=  x_{1b'}x_{2c}x_{3d} = y_{abc'd'}x_{0a} \implies y_{abcd}=y_{abc'd'},
\]
which means that in the corresponding (nonempty) open set of $Z'$ the equations $y_{abcd}=y_{abc'd'}$ of $\Fix_{(+,+)}(\alpha_2\alpha_3\beta_2\beta_3)$ are redundant. Hence
\mbox{$\Fix_{(+,+)}(\alpha_2\alpha_3\beta_2\beta_3)\cap V$} is positive-dimensional. Since $T$ is obtained from $V$ by taking a hypersurface section, we deduce that $\Fix_{(+,+)}(\alpha_2\alpha_3\beta_2\beta_3)\cap T$ is nonempty.
We conclude that the fixed points of $\alpha_2\alpha_3\beta_2\beta_3$ do not contribute to $D_1$
and that their contribution to $\Delta_1$ is $\set{n_1}$.

\medskip
\noindent
The arguments we have used so far can be used to show that the $5$ remaining elements of
$[\alpha_1\beta_1]$ contribute neither to $\Delta_1$ nor to $D_1$.
\end{proof}

The cubic surface with $3$ nodes $S_3$ contains exactly $12$ lines, as represented in Figure~\ref{fig: intersections of A,A',B,B',C,C'} (courtesy of \cite{cubics}).
\begin{figure}[ht]
\includegraphics{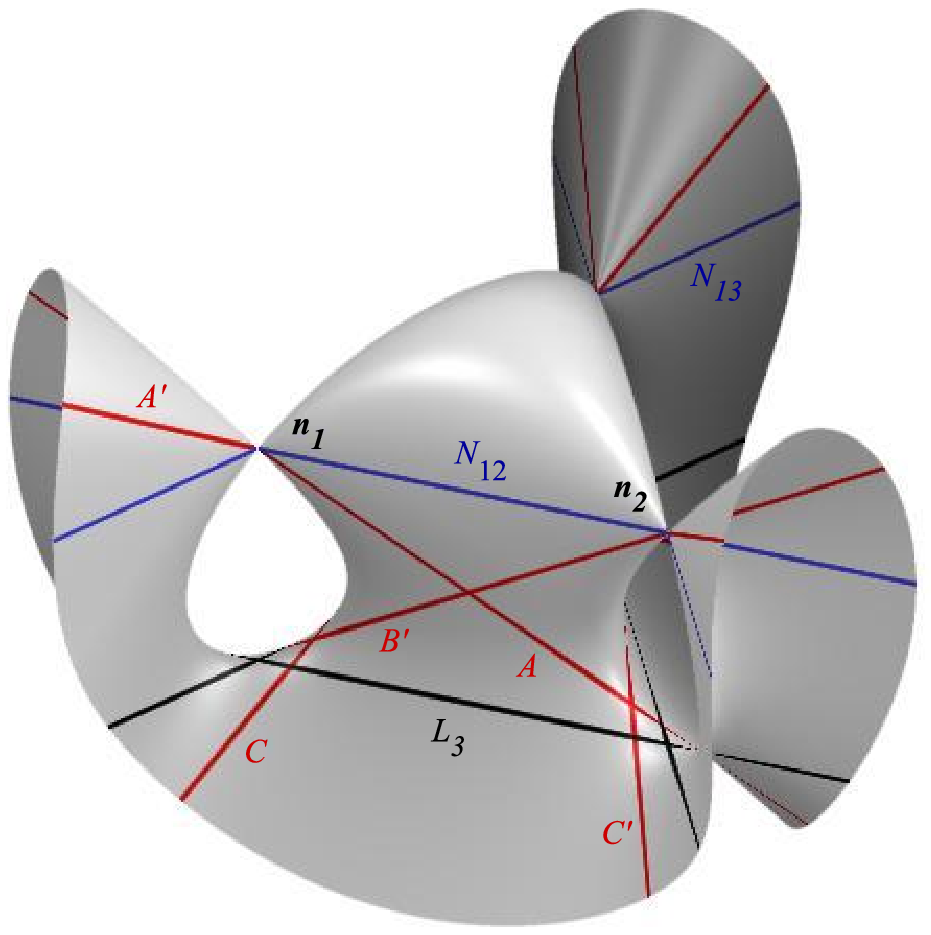}
\caption{Lines on $S_3$}\label{fig: intersections of A,A',B,B',C,C'}
\end{figure}
The plane  $(s_0=0)$ cut the lines $L_1,L_2,L_3$, forming the black triangle in the picture. 
The plane through the nodes cut the lines 
\begin{equation}\label{eq: equations of N_ij}
N_{ij}=(s_0-s_k=l=0),
\end{equation}
where $\set{i,j,k}=\set{1,2,3}$, forming the blue triangle. Each $L_i$ intersects exactly one of the $N_{ij}$: $L_1$ intersects $N_{23}$,
$L_2$ intersects $N_{13}$ and $L_3$ intersects $N_{12}$. 
Consider the plane $\mathcal{P}_1$ through $L_1$ and $n_1$, given by
\begin{equation}\label{eq: plane through L1 and n1}
(\nu_0+\nu_2+\nu_3)s_0 + \nu_1s_1 =0,
\end{equation}
and the analogous planes $\mathcal{P}_2$ through $L_2$ and $n_2$ and
$\mathcal{P}_3$ through $L_3$ and $n_3$.
$S_3\cap \mathcal{P}_i$ splits as
union of $3$ lines: $S_3\cap \mathcal{P}_i=L_1 \cup A \cup A'$,
$S_3\cap \mathcal{P}_2 = L_2 \cup B \cup B'$ and $S_3\cap \mathcal{P}_3 = L_3 \cup C \cup C'$.
We have labeled these last 6 lines as in Figure~\ref{fig: intersections of A,A',B,B',C,C'}, so that
$A,B,C$ are pairwise disjoint.
\medskip

Let $\zeta \colon \Sigma \rightarrow S_3$ be the the blow-up of the $3$ nodes,
and let $E_i$ denote the  \mbox{exceptional} divisor of $n_i$.
With abuse of notation, let us denote by $A$, $A'$, $B$, $B'$, $C$, $C'$, $L_i$, $N_{ij}$
the strict transforms in $\Sigma$
of the namesake lines. Similarly we do not change the notation for the strict
transforms in $\Sigma$ of $C_i\subset S_3$.
Denote by $H_{\Sigma}$ be the pull-back of an hyperplane section.
Since  $K_{\Sigma}=-H_{\Sigma}$, the strict transform of every line in $S_3$ is a $(-1)$-curve.
The curves $N_{12}$, $N_{13}$, $N_{23}$, $A'$, $B'$, $C'$ are pairwise disjoint rational curves with self-intersection
$-1$; by Castelnuovo's criterion we can contract them to a smooth rational surface with $K^2=9$.
Therefore, the contraction of these curves yields a morphism $\xi \colon \Sigma \rightarrow \PP^2$.
Again, with abuse of notation, we shall continue using the same notation for a curve in $S_3$,
its strict transform in $\Sigma$, and, when it is not contracted to a point, its image in $\PP^2$.
Let us denote by $r_{12},r_{13},r_{23}$ and by $\xx_1,\xx_2,\xx_3$, the points of $\PP^2$ to which $\xi$ contracts
$N_{12},N_{13},N_{23}$ and $A',B',C'$, respectively. In $\PP^2$ we get the configuration of
curves of Figure~\ref{fig: configuration in the plane}.
\begin{figure}[h]\label{fig: twc}
\begin{center}
\includegraphics{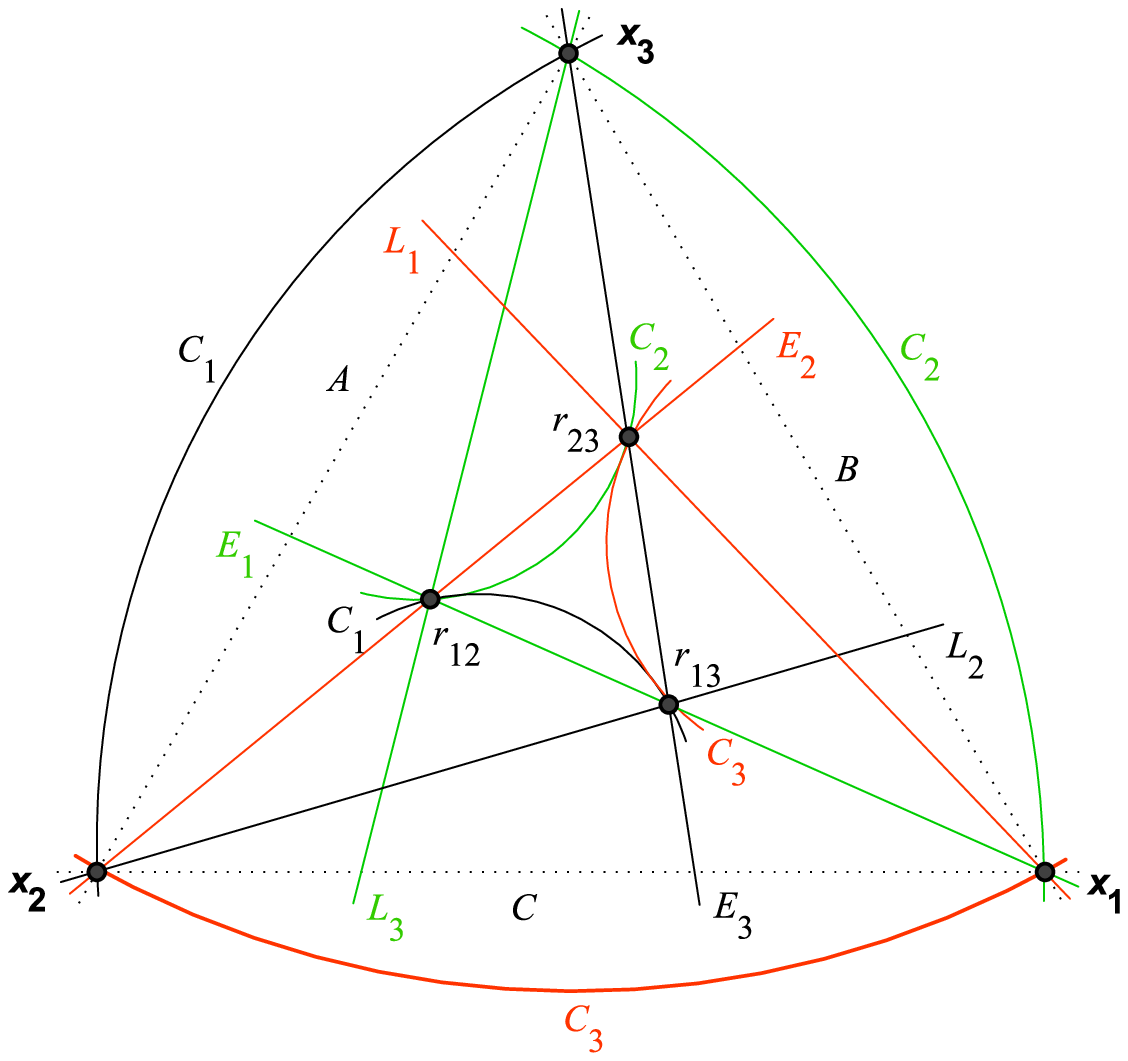}
\end{center}
\caption{The branch divisors of $\gamma'' \colon S'' \rt \PP^2$}\label{fig: configuration in the plane}
\end{figure}
We leave to the reader the straightforward check that $L_1,L_2,L_3$, $E_1,E_2,E_3$, $A$, $B$, $C$ in $\PP^2$ are in the
configuration of Figure~\ref{fig: configuration in the plane}. As for $C_2$, using the equations of
$C_2$ and $N_{12}$, (\ref{eq: equations of L_i and C_i}) and (\ref{eq: equations of N_ij}), we see that it meets $N_{12}$ in $S_3$.
Similarly $C_2$ meets $N_{23}$. Hence in the plane $C_2$ contains the points $r_{12}$, $r_{23}$.
To see that $C_2$ contains the points $\xx_1$ and $\xx_3$ it is enough to show that, in $S_3$,
$C_2$ meets the lines $A,A'$ and the lines $C,C'$.
Indeed, as $C_2=H_{S_3}-L_2$ and $A+A'=H_{S_3}-L_1$ in $S_3$, we have
$C_2(A+A')=(H_{S_3}-L_2)(H_{S_3}-L_1)=H^2_{S_3}-H_{S_3}L_1-H_{S_3}L_2+L_1L_2 = 2$.
Likewise one shows that $C_2(C+C')=2$. Additionally, $C_2(B+B')=0$,
hence $C_2$ does not contain $\xx_2$. The conics $C_1$ and $C_3$ have similar properties, obtained by cyclic
permutation of the indices $\set{1,2,3}$.

\begin{remark}\label{rmk: the general branch divisors}
Consider the following commutative diagram:
\begin{equation}\label{eq: diag:SS'S''}
\xymatrix{
S\ar[d]_\gamma&S'\ar[d]_{\gamma'}\ar[r]^{\hat{\xi}}\ar[l]_{\hat{\zeta}}&S''\ar[d]^{\gamma''}\\
S_3&\Sigma\ar[r]_{\xi}\ar[l]^{\zeta}&\PP^2\\
}
\end{equation}
where $S'$, $S''$, $\gamma'$, $\gamma''$, $\hat{\zeta}$, $\hat{\xi}$ are constructed as fiber products
to make both squares cartesian.
Notice that $\hat{\xi}$ is the contraction of the preimages of $N_{12}$, $N_{13}$, $N_{23}$,
$A'$, $B'$, $C'$. All horizontal maps in (\ref{eq: diag:SS'S''}) are birational morphisms
and all vertical maps are bidouble covers. Consider each of the
non-trivial involutions of $\Gamma$, $\theta_i$, and denote by $D_i'$ and $D_i''$ the images via $\gamma'$ and $\gamma''$ of the fix locus of $\theta_i$.
According to Proposition~\ref{prop: the branch locus of the bicanonical map}, $\theta_i$ fixes each of the $2$ pre-images of the node $n_i$. 
Since these are smooth isolated fixed points for $\theta_i$, $\theta_i$ fixes each point in the 
exceptional divisor of their blow up and accordingly $E_i$ is in the branch divisor of $\gamma''$ associated with $\theta_i$.
In conclusion, we have
\begin{equation}\label{eq: branch divisors for general T in N}
D''_1=E_1+C_2+L_3\ \ \ \
D''_2=E_2+C_3+L_1\ \ \ \
D''_3=E_3+C_1+L_2.
\end{equation}
In the Figure~\ref{fig: configuration in the plane} we have depicted the divisor $D''_1$ in green and in red the divisor $D''_2$.
Note that $S''$ is singular and $\hat{\xi}$ is a resolution of its singularities.
\end{remark}

\begin{theorem}\label{thm: The Burniat pencil}
A general surface in the $1$-dimensional linear subsystem
\[
\mathcal{B}=\set{T=V\cap(q=0)\mid (\nu_0,\nu_1,\nu_2,\nu_3,\nu_4)=(-\nu,\nu,\nu,\nu,\nu_4), \text{ for }(\nu,\nu_4)\in \PP^1}\subset \mathcal{N}
\]
is a surface with $24$ isolated rational double points as only singularities and $G$ acts freely on it. The quotient
$T/G$ is the canonical model of a tertiary Burniat surface and, conversely, every tertiary Burniat surface arises in this way.
\end{theorem}

\begin{proof}
Analyzing the base locus of $\mathcal{B}$, like in the proof of Theorem~\ref{thm: free action on $T$}, we get
$\Sing(T)\subset (l=\sum_{abcd\in\LL} (-1)^a y_{abcd}=0)$, where $l=\nu(-s_0+s_1+s_2+s_3)$.
Let $T_1,T_2\in \mathcal{B}$ be given in $V$ by $s_0-s_1-s_2-s_3 =0$ and $\sum_{abcd\in\LL}(-1)^a y_{abcd}=0$.

\medskip

\noindent
Fix coordinates $x_1,x_2,x_3$ on $\CC^3$. Consider the open set $\Omega_{00}= V\setminus\set{x_{00}=0}$ and the map $\xi_1\colon \Omega_{00} \rt \CC^3$ given by
\begin{equation}\label{eq: local iso}
\xi_1(\xx,\yy)=\pare{\frac{x_{10}}{x_{00}},\frac{x_{20}}{x_{00}},\frac{x_{30}}{x_{00}}}.
\end{equation}
 Since $x_{00}x_{01}\not =0$ and (\ref{eq: quadrics}) hold,  $\xi_1(\Omega_{00})$ is
contained in $\CC^3\setminus \set{x_1x_2x_3=0}$ and the map $\xi_2 \colon \CC^3\setminus \set{x_1x_2x_3=0} \rt \Omega_{00}$
given by
\begin{equation}\label{eq: definition of xi_1}
\xi_2(x_1,x_2,x_3)=(1,-1,x_1,-1/x_1,x_2,-1/x_2,x_3,-1/x_3,\dots),
\end{equation}
is the inverse of $\xi_1$.
Let $\mathcal{B}_{\Omega_{00}}$ denote the  pencil $\set{T\cap \Omega_{00} : T\in \mathcal{B}}$ on $\Omega_{00}$.
The pencil $\xi_2^* \mathcal{B}_{\Omega_{00}}$ is spanned by
$\xi_1(T_1\cap \Omega_{00})$ and $\xi_1(T_2\cap \Omega_{00})$, whose equations are: 
\begin{equation}\label{eq: pencil eq in omega00}
\renewcommand{\arraystretch}{1.4}
\begin{array}{c}
F_1 := 1-\frac{1}{2}\sum_{i>0}\pare{x_{i}^2+1/x_{i}^2}=0,\\
F_2 := (x_{1}-1/x_{1})(x_{2}-1/x_{2})(x_{3}-1/x_{3})=0,
\end{array}
\end{equation}
respectively. We show next that a general member of  $\xi_2^* \mathcal{B}_{\Omega_{00}}$ is smooth outside a (fixed) set of $24$ rational double points.
Since $\frac{\partial F_1}{\partial x_i}=1/x_i^3-x_i$, the singularities of $\xi_1(T_1\cap \Omega_{00})$
lie in the set defined by $x_1^4=x_2^4=x_3^4=1$. These equations define $64$ points. However only the $24$ points of the set
\begin{equation}\label{eq: the set of 24 nodes}
\mathfrak{D}=\set{(\pm \epsilon,\pm 1,\pm 1),  (\pm 1,\pm \epsilon,\pm 1),  (\pm 1,\pm 1,\pm \epsilon)},
\end{equation}
where $\epsilon$ is a square root of $-1$, actually belong to $\xi_1(T_1\cap \Omega_{00})$. As $\frac{\partial^2 F_1}{\partial x_i\partial x_j}=0$, for $i\not = j$, and
$\frac{\partial^2 F_1}{\partial x_i^2}=-3/x_i^4-1$, we see that the determinant of the Hessian matrix is nonzero at the points
of $\mathfrak{D}$, showing that they are indeed ordinary double points of $\xi_1(T_1\cap \Omega_{00})$. For every point of $\mathfrak{D}$, two factors of $F_2$
vanish, thus it is clear that $\xi_1(T_2\cap \Omega_{00})$ is also singular at the points of $\mathfrak{D}$. This shows that a general member
of $\xi_2^*\mathcal{B}_{\Omega_{00}}$ has a rational double point at each point of $\mathfrak{D}$. Since $\xi_1(T_1\cap\Omega_{00})$
is smooth away from $\mathfrak{D}$ it follows that a general member of $\xi_2^*\mathcal{B}_{\Omega_{00}}$ is also smooth away from $\mathfrak{D}$.
\medskip

\noindent
We proceed to show that a general $T\in \mathcal{B}$ is smooth along \mbox{$T\cap (x_{00}=0)=T\setminus \Omega_{00}$}.
If $x_{00}=0$, then $x_{01}=-x_{00}=0$ and, from (\ref{eq: quadrics}), $x_{10}x_{11}=x_{20}x_{21}=x_{30}x_{31}=0$.
Let $\Omega_{ab}^{ij}$ denote the open set of $V$ given by $x_{ia}x_{jb}\not = 0$.
Since a general $T\in \mathcal{B}$ has $\Sing(T)\subset (l=0)$, if all the variables $x_{00},x_{01},\dots,x_{30},x_{31}$
but one vanish at a point of $\Sing(T)\cap(x_{00}=0)$, then from $l=0$ we deduce the remaining one must vanish also.
From this, using (\ref{eq: quartic relations}) and the fact that $\Sing(T)$ is contained in $T_2=(\sum_{abcd\in\LL}(-1)^a y_{abcd}=0)$,
we deduce that all variables must vanish, which is impossible. Hence, for a general $T\in \mathcal{B}$, there exist $i,j,a,b$
with $j>i>0$ such that $\Sing(T)\cap (x_{00}=0)\subset \Omega_{ab}^{ij}$. Since the role that $x_{10},x_{11},x_{20},x_{21},x_{30},x_{31}$
play in the equations of $V$ and $T\in \mathcal{B}$ is symmetric, we may reduce to showing that a general member $T\in \mathcal{B}$ is
smooth along $T\cap(x_{00}=0)\cap \Omega^{23}_{11}$. Similarly to what we did earlier, we consider a map
$\zeta_1 \colon \Omega_{11}^{23}\rt \PP(1^3,2)$ given by $\zeta_1(\xx,\yy)=(x_{00},x_{21},x_{31},y_{0000})$.
This map has image the (affine) open set defined by $x_{21}x_{31}y_{0000}\not = 0$. This is a consequence of the
quartic relation $y_{0000}y_{1100}-x_{21}^2x_{31}^2$ in (\ref{eq: quartic relations}) which holds in $V$. As before,
to show that an inverse $\zeta_2 \colon \PP(1^3,2)\setminus (x_{21}x_{31}y_{0000}\not = 0)\rt \Omega_{11}^{23}$ to
$\zeta_1$ exists, it is enough to express every variable on $\Omega_{11}^{23}$ has a rational function of $x_{00},x_{21},x_{31},y_{0000}$.
Using the equations of $V$, i.e., (\ref{eq: quadrics}), (\ref{eq: cubic relations}), (\ref{eq: quartic relations}) and $x_{00}+x_{01}=0$,
on $\Omega_{11}^{23}$ we have:
\begin{equation}\label{eq: L1021}
\renewcommand{\arraystretch}{1.8}
\begin{array}{l}
\displaystyle x_{01}=-x_{00},\quad y_{1100} = \frac{x_{21}^2x_{31}^2}{y_{0000}},\\
\displaystyle x_{20}=\frac{x_{00}x_{01}}{x_{21}}=-\frac{x_{00}^2}{x_{21}},\quad x_{30}=\frac{x_{00}x_{01}}{x_{31}}=-\frac{x_{00}^2}{x_{31}},\\
\displaystyle x_{10}=\frac{x_{01}x_{21}x_{31}}{y_{0000}}=-\frac{x_{00}x_{21}x_{31}}{y_{0000}}, \quad
x_{11}=\frac{x_{00}x_{21}x_{31}}{y_{1100}}=\frac{x_{00}y_{0000}}{x_{21}x_{31}}.
\end{array}
\end{equation}
which are all rational functions of $x_{00},x_{21},x_{31},y_{0000}$. Moreover
\begin{equation}\label{eq: L1032}
\displaystyle
y_{abc1} = \frac{x_{0a'}x_{1b'}x_{2c'}}{x_{31}}\quad \text{and} \quad y_{ab1d} = \frac{x_{0a'}x_{1b'}x_{3d'}}{x_{21}},
\end{equation}
which, using (\ref{eq: L1021}), can be seen to be also rational functions of $x_{00},x_{21},x_{31},y_{0000}$.
\medskip

\noindent
Consider $\mathcal{B}_{\Omega_{11}^{23}}$ the pencil $\set{T\cap \Omega_{11}^{23} : T \in \mathcal{B}}$.
Next we show that a general member of
$\zeta_2^*\mathcal{B}_{\Omega_{11}^{23}}$ is smooth along $x_{00}=0$. It suffices to show that $\zeta_1(T_1\cap \Omega_{11}^{23})$
is smooth along $x_{00}=0$. Additionally, since $\zeta_1(T_1\cap \Omega_{11}^{23})$ does not meet the singular point of $\PP(1^3,2)$ we
can reduce to showing quasi-smoothness, or, more precisely, non-vanishing of the Jacobian matrix of the polynomial $F_3$, obtained
from the equation of $\zeta_1(T_1\cap\Omega_{11}^{23})$ by setting $y_{0000}=1$, at the points of $\zeta_1(T_1\cap \Omega_{11}^{23}\cap (x_{00}=0))$.
From (\ref{eq: L1021}) we deduce
$
F_3 = x_{00}^2 x_{21}^2 x_{31}^2 - x_{00}^2 x_{21}^4 x_{31}^4 - x_{00}^2 - x_{00}^4  x_{31}^2 - x_{21}^4 x_{31}^2 - x_{00}^4 x_{21}^2 - x_{21}^2 x_{31}^4
$
using, to ease notation, $x_{00}$, $x_{21}$, $x_{31}$ as coordinates for the corresponding affine piece of $\PP(1^3,2)$.
Hence, at $x_{00}=0$,
\[
\frac{\partial F_3}{\partial x_{21}}=-4x_{21}^3x_{31}^2-2x_{21}x_{31}^4\quad  \text{and}\quad
\frac{\partial F_3}{\partial x_{31}}=-2x_{21}^4x_{31}-4x_{21}^2 x_{31}^3
\]
which have no common zeros for $x_{21}x_{31}\not = 0$.
\medskip

\noindent
We have shown that a general member of $\mathcal{B}$ is a smooth away from a set of $24$ rational
double points given by $\xi_2(\mathfrak{D})$ where $\mathfrak{D}$ is the set of points (\ref{eq: the set of 24 nodes}),
in other words, the set of points given in local coordinates by (\ref{eq: the set of 24 nodes}).
Notice that by Remark~\ref{rmk: needed for Burniat pencil}, the group $G$ acts freely on a general member of $\mathcal{B}$.
To show that $S:=T/G$ for a general $T$ in $\mathcal{B}$ is the canonical model of a Burniat surface we analyze
(\ref{eq: diag:SS'S''}) for this case in detail. We start by observing that if
$(\nu_0,\nu_1,\nu_2,\nu_3,\nu_4) = (-\nu,\nu,\nu,\nu,\nu_4)$ then the plane
$\mathcal{P}_1$ defined in (\ref{eq: plane through L1 and n1}) is nothing other that the
plane $(s_0+s_1=0)$ and hence the conic $C_1$ splits as $A \cup A'$. Likewise, $C_2$ splits up
as $B\cup B'$ and $C_3$ splits up as $C\cup C'$. Recall that $n_1\in A\cap A'$,  $n_2\in B\cap B'$
and $n_3\in C\cap C'$. Also, the nodes become $n_1=(1,-1,1,1)$, $n_2=(1,1,-1,1)$ and $n_3=(1,1,1,-1)$.
Their pre-image in $T$ coincides with the $24$ ordinary nodes of $T$. Indeed we see that $8$ points of $T$, written in local coordinates
of $\Omega_{00}\simeq \CC^3\setminus (x_1x_2x_3=0)$ as $(\pm \epsilon, \pm 1, \pm 1)$, map to $n_1$; the $8$ points $(\pm 1, \pm \epsilon, \pm 1)$ map to $n_2$ and
the $8$ points $(\pm 1, \pm 1, \pm \epsilon)$ map to $n_3$. Since $G$ acts freely on $T$, each of these sets of $8$ points maps to a single
point in $S:=T/G$  which is a node of $S$ and is fixed by every element of $\Gamma$.
Denote these $3$ nodes of $S$ by $\hat{n}_1,\hat{n}_2,\hat{n}_3$.
Using the notation of Proposition~\ref{prop: the branch locus of the bicanonical map}, we claim that
\[\renewcommand{\arraystretch}{1.3}
\begin{array}{l}
\Delta_1 = \set{n_1,n_{3}},\quad D_1=B+B'+L_3, \\
\Delta_2 = \set{n_2,n_{1}},\quad D_2=C+C'+L_1, \\
\Delta_3 = \set{n_3,n_{2}},\quad D_3=A+A'+L_2.
\end{array}
\]
We compute $\Delta_1$ and $D_1$
by analyzing the fixed loci on $T$ of the elements of:
\[
[\alpha_1\beta_1]=\set{\alpha_1\beta_1, \beta_1\beta_2, \alpha_1\alpha_2\beta_1\beta_3, \alpha_1\alpha_3, \alpha_2\beta_1\beta_2\beta_3,
\alpha_3\beta_2, \alpha_1\alpha_2\alpha_3\beta_3,\alpha_2\alpha_3\beta_2\beta_3}.
\]
The computation of $\Delta_2$, $D_2$, $\Delta_3$ and $D_3$ follows by symmetry.
Recall from the proof of Proposition~\ref{prop: the branch locus of the bicanonical map},
that $\Fix(\beta_1\beta_2)\cap T$ maps to $L_3$ on $S_3$, $\Fix(\alpha_2\beta_1\beta_2\beta_3)\cap T$
maps to $C_2$ which is now $B\cup B'$ and that $\Fix(\alpha_2\alpha_3\beta_2\beta_3)\cap T$ maps to $n_1$.
With the assumptions of that proposition, all of the other elements of $[\alpha_1\beta_1]$ have empty fixed locus on $T$.
In case of $T\in \mathcal{B}$ this is no longer true.
Indeed all but $\alpha_1\alpha_2\alpha_3\beta_3$ have empty fixed locus on $T$.
To see this, recall that $\Fix_{(+,+)}(\alpha_1\alpha_2\alpha_3\beta_3)$ and $\Fix_{(-,+)}(\alpha_1\alpha_2\alpha_3\beta_3)$ are given by:
\[
\begin{array}{c}
(x_{00}-x_{01}=x_{10}-x_{11}=x_{20}-x_{21}=x_{30}+x_{31}=y_{abcd}+y_{a'bc'd}=0,\forall_{abcd \in \LL}), \\
(x_{00}+x_{01}=x_{10}+x_{11}=x_{20}+x_{21}=x_{30}-x_{31}=y_{abcd}-y_{abc'd'}=0,\forall_{abcd \in \LL}), \\
\end{array}
\]
respectively. It is easy to see that $\Fix_{(+,+)}(\alpha_1\alpha_2\alpha_3\beta_3)\cap T$ is empty.
However the locus $\Fix_{(-,+)}(\alpha_1\alpha_2\alpha_3\beta_3)$ now contains the set of points $\set{(\pm 1,\pm 1,\pm \epsilon)}$, given in
local coordinates in $\Omega_{00}$. Since $n_1$ and $n_3$ are the only isolated fixed points of $\theta_1$ (notice that $n_2\in B\cap B'$) 
we have $\Delta_1=\set{n_1,n_3}$ and $D_1=B+B'+L_3$.
\medskip

\noindent
We claim that the branch loci of $\gamma'\colon S' \rt \Sigma$ are:
\[\renewcommand{\arraystretch}{1.3}
\begin{array}{l}
D'_1=E_3+B+B'+L_3, \\
D'_2=E_2+C+C'+L_1, \\
D'_3=E_1+A+A'+L_2.
\end{array}
\]
Again by symmetry it is enough to compute $D'_1$. To do this we must analyze the action of
$\theta_1$ on the tangent cone at $\hat{n}_1,\hat{n}_2$ and $\hat{n}_3$, showing that it fixes every tangent direction
in the tangent cone at $\hat{n}_3$ and that it does not act in this way on the tangent cones at the nodes $\hat{n}_1$ and $\hat{n}_2$.
This will mean that the action of $\theta_1$ on $S'$ will fix pointwise $E_3$ and will not fix pointwise $E_1$ and $E_2$.
(Recall that $S'$ can be obtained by blowing up $\hat{n}_1,\hat{n}_2,\hat{n}_3$.)
To analyze the action of $\theta_1$ on each of the nodes we will study the action of the corresponding involutions
of $[\alpha_1\beta_1]$ on the local model given by $\Omega_{00}\subset V$.
As we showed earlier, the involutions in $[\alpha_1\beta_1]$
which fix in
$T$ points in the pre-image of $\hat{n}_1,\hat{n}_2,\hat{n}_3$ are $\alpha_2\alpha_3\beta_2\beta_3$, $\alpha_2\beta_1\beta_2\beta_3$ and
$\alpha_1\alpha_2\alpha_3\beta_3$. We have seen that $\alpha_2\beta_1\beta_2\beta_3$ fixes a positive dimensional locus containing
the pre-images of $\hat{n}_2$, hence, locally at each pre-image, this involution cannot fix every tangent direction to it and therefore $E_2$
is not in $D'_1$. The involution $\alpha_2\alpha_3\beta_2\beta_3$, whose fixed locus on $T$ maps to $\set{\hat{n}_1}$,
can be written in the local model $\CC^3\setminus (x_1x_2x_3=0)\simeq \Omega_{00}$ as:
\[
(x_1,x_2,x_3) = \pare{\frac{x_{10}}{x_{00}},\frac{x_{20}}{x_{00}},\frac{x_{30}}{x_{00}}}\mapsto \pare{\frac{x_{10}}{x_{00}},-\frac{x_{21}}{x_{00}},-\frac{x_{31}}{x_{00}}} = \pare{{x_1},\frac{1}{x_2},\frac{1}{x_3}}.
\]
(Recall that $x_{00}+x_{01}=0$ and $x_{00}x_{01}=x_{i0}x_{i1}\implies x_{00}/x_{i0}=x_{i1}/x_{01}$.)
We see that $(\pm \epsilon,\pm 1,\pm 1)$ are fixed. However we see also that the fixed loci of this involution in the
ambient $\CC^3\setminus (x_1x_2x_3=0)$ is a set of four lines, going through the $8$ points $(\pm \epsilon, \pm 1,\pm 1)$.
Hence this involution does not fix all of the tangent directions at any of these points. This implies that $\theta_1$ does
not fix all of the tangent directions of $\hat{n}_1$ and thus $E_1$ is also not in $D'_1$. Finally, writing
$\alpha_1\alpha_2\alpha_3\beta_3$ in the local model:
\[
(x_1,x_2,x_3) = \pare{\frac{x_{10}}{x_{00}},\frac{x_{20}}{x_{00}},\frac{x_{30}}{x_{00}}}\mapsto \pare{\frac{x_{11}}{x_{01}},\frac{x_{21}}{x_{01}},-\frac{x_{31}}{x_{01}}} = \pare{\frac{1}{x_1},\frac{1}{x_2},-\frac{1}{x_3}},
\]
we see that this involution fixes exactly the set of points given in local coordinates by $\set{(\pm 1, \pm 1, \pm \epsilon)}$.
This coincides with the pre-image of $\set{\hat{n_3}}$. Moreover, since it fixes only finitely many points, it must
fix every tangent direction at each of these points. We conclude that $\theta_3$ fixes every tangent direction of ${\hat n}_3$
and thus $D'_1=E_3+B+B'+L_3$.
\medskip

\noindent
It is now easy to compute the ramification divisors of $\gamma'' \colon S'' \rt \PP^2$. Since $S''$ can be obtained
by contracting the pre-images of $N_{23},N_{13},N_{12},A',B',C'$ it is clear that
\begin{equation}\label{eq: branch divisors for Burniat pencil}
\renewcommand{\arraystretch}{1.3}
\begin{array}{l}
D''_1=E_3+B+L_3, \\
D''_2=E_2+C+L_1, \\
D''_3=E_1+A+L_2.
\end{array}
\end{equation}
With the help of Figure~\ref{fig: configuration in the plane} we see that these are exactly the branch loci for a
tertiary Burniat surface.
\medskip

\noindent
The space of tertiary Burniat surfaces is parameterized by $\lambda \in \CC^*\setminus \{1\}$ as follows. In Figure~\ref{fig: BurniatPicture},
we may always choose coordinates $(u_0,u_1,u_2)$ such that $\xx_1=(1,0,0)$, $\xx_2=(0,1,0)$, $\xx_3=(0,0,1)$ and the further $3$
marked points are respectively $(1,1,1)$, $(1,1,\lambda)$ and $(\lambda,1,\lambda)$.
The bicanonical image of the Burniat surface is the image of $\PP^2$ in $\PP^3$ by the linear system of
cubics through the $6$ marked points. If we choose, as basis for this system, the cubics:
\[\renewcommand{\arraystretch}{1.2}
\begin{array}{l}
s_0=-\frac12 (u_0-\lambda u_1)(u_1-u_2)(u_2-\lambda u_0)\\
s_1=(1-\lambda)u_0(u_1-u_2)(\lambda u_1-u_2)-s_0\\
s_2=(1-\lambda)u_1(u_2-u_0)(u_2-\lambda u_0)-s_0\\
s_3=(1-\lambda)u_2(u_0-u_1)(u_0-\lambda u_1)-s_0\\
\end{array}
\]
then, one can check that
\[
(\lambda +1)^2 (s_1-s_0)(s_2-s_0)(s_3-s_0)= -2\lambda s_0(s_3+s_2+s_1-s_0)^2
\]
and we easily conclude that the tertiary Burniat surface under consideration is isomorphic to $S=T/G$ with $T\in\mathcal{N}$ given by
$-\nu_0=\nu_1=\nu_2=\nu_3=\sqrt{-\lambda}$ and $\nu_4=4(\lambda+1)$.
\end{proof}

\end{document}